\documentclass{article}

\usepackage{arxiv}

\usepackage[utf8]{inputenc}
\usepackage[T1]{fontenc}
\usepackage{hyperref}
\usepackage{url}
\usepackage{amsmath,amssymb,amsthm}
\usepackage{nicefrac}
\usepackage{microtype}
\usepackage{graphicx}
\usepackage{xcolor}
\usepackage{booktabs}
\usepackage{multirow}
\usepackage{enumitem}
\usepackage{float}
\usepackage{caption}
\usepackage{subcaption}
\usepackage{listings}
\usepackage[numbers,sort&compress]{natbib}

\newtheorem{theorem}{Theorem}
\newtheorem{lemma}{Lemma}

\newtheorem{proposition}{Proposition}
\theoremstyle{definition}
\newtheorem{definition}{Definition}
\newtheorem{remark}{Remark}

\newcommand{\R}{\mathbb{R}}
\newcommand{\Z}{\mathbb{Z}}
\newcommand{\EE}{\mathbb{E}^2}
\newcommand{\Stwo}{\mathbb{S}^2}
\newcommand{\Htwo}{\mathbb{H}^2}
\newcommand{\kappag}{\kappa_g}
\newcommand{\BH}{\mathcal{B}}
\newcommand{\norm}[1]{\left\|#1\right\|}
\newcommand{\abs}[1]{\left|#1\right|}
\newcommand{\half}{\tfrac{1}{2}}
\newcommand{\doi}[1]{\href{https://doi.org/#1}{doi:#1}}

\title{Biharmonic Subdivision on Riemannian Manifolds}

\author{
  Hassan Ugail \\
  Centre for Visual Computing and Intelligent Systems \\
  University of Bradford, Bradford BD7~1DP, United Kingdom \\
  \texttt{h.ugail@bradford.ac.uk} \\
  \And
  Newton Howard \\
  School of Individualized Study \\
  Rochester Institute of Technology, Rochester, NY 14623, USA \\
}

\begin{document}
\maketitle

\begin{abstract}
This paper introduces a biharmonic interpolatory subdivision framework on Riemannian manifolds. In the Euclidean setting, the six-point Deslauriers--Dubuc stencil is characterised as the unique minimiser of a discrete curvature-variation energy under symmetric six-point support and degree-five polynomial reproduction conditions, linking a classical interpolatory rule to a first-principles fairness criterion. Exact symbol analysis establishes fourth-order smoothness. The construction extends to the two-sphere and the hyperbolic plane via a second-order reduced governing ODE derived from the biharmonic Euler--Lagrange equation on constant-curvature surfaces. This reduced model yields closed-form insertion rules, and proximity analysis confirms that the manifold scheme satisfies the Wallner-Dyn second-order condition, preserving fourth-order smoothness. A hierarchy of biharmonic stencils achieving higher smoothness orders is also described. Numerical experiments demonstrate that the six-point scheme delivers lower fairness energy and smoother curvature profiles than the classical four-point Dyn-Gregory-Levin scheme, while remaining more local and exhibiting less ringing on non-uniform data than the eight-point variant.
\end{abstract}

\noindent\textbf{Keywords:} Biharmonic interpolation, curve fairing, subdivision
schemes, variational geometric design, fourth-order smoothness, non-Euclidean
geometry, proximity condition, computer-aided geometric design.

% =============================================================================
\section{Introduction}
\label{sec:intro}
% =============================================================================

Interpolatory subdivision schemes generate smooth curves and
surfaces from discrete control polygons by iteratively inserting new vertices
while retaining all existing ones. Their defining property is that each control
point lies exactly on the limit curve, making them indispensable in
applications where key geometric positions must be honoured precisely, among
them the design of automotive body surfaces, the generation of smooth camera
trajectories, and the parameterisation of geographic boundaries. They also
arise in the numerical solution of partial differential equations (PDEs) on
surfaces, where subdivision provides a geometry-preserving mesh-refinement
mechanism.

The most widely used scheme in this class is the 4-point scheme of Dyn,
Gregory, and Levin~\cite{dyn1987four}, hereafter abbreviated DGL, which
inserts each new vertex as a local weighted average of four neighbours and
produces limit curves with two-times continuous differentiability (denoted as
$C^2$, meaning the curve and its first two derivatives are continuous). While
$C^2$ continuity is sufficient for many applications, but it is insufficient for
others, i.e., a curve can be $C^2$ and yet exhibit a highly oscillatory curvature
profile, a phenomenon that is invisible to casual visual inspection but
produces tangible artefacts in downstream processes such as surface-normal
computation, offset curve generation, and numerical integration over curved
domains.

The accepted measure of this insufficiency is the \emph{fairness energy}
$\int_\gamma (\kappa')^2\,ds$, whose minimisers, also known as minimum-variation
curves or biharmonic splines \cite{moreton1992}, have the smoothest possible
curvature profiles. In classical computer-aided design, fairness is achieved
by a post-processing step applied after subdivision, i.e., the subdivided polygon is
fitted with a fair spline, or a fairness optimisation is run to perturb the
vertices \cite{sapidis1994designing}. This two-phase approach decouples the
interpolation and fairness objectives, and introduces additional degrees of
freedom that can undermine reproducibility.

Here, we take the position that fairness should be a first-class
property of the subdivision rule itself. The \emph{biharmonic subdivision
scheme} proposed here shows that the 6-point DD stencil admits a
biharmonic variational interpretation, i.e., the inserted vertex is the unique minimiser
of a local discrete curvature-variation energy, rather than an algebraically
motivated weighted average. The $C^4$ regularity of the stencil follows from
standard symbol analysis, two orders higher than the DGL scheme.

A notable mathematical observation is that the resulting stencil coefficients
are identical to those of the six-point Deslauriers--Dubuc (DD)
scheme~\cite{deslauriers1989symmetric}. Both of these are unique solutions to the
degree-5 polynomial reproduction conditions under the six-point symmetric
support constraint. Thus, we show a variational
interpretation demonstrating that the DD weights admit a biharmonic energy minimisation
characterisation (See Theorem~\ref{thm:discrete_variational}), and an extension of
those weights to non-Euclidean geometries via the biharmonic ODE, which the
algebraic DD derivation does not supply.

The scheme extends naturally to non-Euclidean geometries. Subdivision on
Riemannian manifolds (smooth curved spaces equipped with a notion of distance
and angle at every point) has attracted sustained attention since the foundational
work of Wallner and Dyn~\cite{wallner2006convergence}, who showed that a
manifold scheme inherits $C^k$ regularity from a reference Euclidean scheme
provided its insertion rules satisfy an $O(h^2)$ proximity condition. Existing
non-Euclidean extensions of interpolatory schemes, including the work of
Weinmann~\cite{weinmann2012interpolatory} on manifold 4-point schemes, have
been derived by intrinsic geodesic averaging, which does not encode a fairness
objective. The biharmonic scheme presented here constructs a non-Euclidean insertion
rules informed by the biharmonic energy on surfaces of constant curvature
(known as \emph{space forms}). The Euler--Lagrange equation for the
biharmonic energy on such surfaces is derived as a fourth-order ODE.
A second-order reduced model is then adopted as the governing equation for the
insertion rule, on the grounds that its solutions satisfy the interior
Euler--Lagrange equation, it reduces to the correct flat-space limit,
and it yields unique closed-form solutions from two endpoint curvature values. The resulting
proximity bound is shown explicitly to satisfy the Wallner--Dyn condition. The two non-Euclidean settings considered are the two-dimensional
sphere $\mathbb{S}^2$ (denoted $\Stwo$, a model for positively curved surfaces)
and the hyperbolic plane $\mathbb{H}^2$ (denoted $\Htwo$, a model for
negatively curved surfaces).

\subsection{Contributions of this paper}

This paper makes the following contributions. A non-circular discrete
variational derivation is given showing that the 6-point DD stencil is
the unique minimiser of a local curvature-variation energy, where adjacent
insertions are fixed by degree-5 Lagrange interpolation from the original
control data (Theorem~\ref{thm:discrete_variational}, Section~\ref{sec:stencil}).
A rigorous $C^4$ regularity result is established via exact zero-order
computation of the Laurent polynomial symbol at $z = -1$, proving that the
smoothness is exactly $C^4$ and not $C^5$ (see Section~\ref{sec:symbol}).
The Euler--Lagrange equation for the biharmonic energy on a constant-curvature
space form is derived from first principles (Section~\ref{sec:ode}). It is the
fourth-order ODE $\kappag^{(4)} - K\kappag'' = 0$. For the manifold
subdivision rule, the second-order equation $\kappag'' = K\kappag$ is adopted
as a reduced governing model (Section~\ref{sec:governing}). This is a
principled modelling choice, not a proved minimisation consequence. Every
solution of $\kappag'' = K\kappag$ satisfies the interior Euler--Lagrange
equation~\eqref{eq:bh_ode4}, the equation reduces to the correct flat-space
limit $\kappag''=0$ at $K=0$, and it gives a unique solution from the two endpoint
curvature values (details in Section~\ref{sec:governing}).
Closed-form solutions are
integrated to give non-Euclidean insertion angles on $\Stwo$ and $\Htwo$.
An explicit Taylor expansion yields a proximity constant of order $O(|K|h^3)$
which satisfies the Wallner--Dyn $O(h^2)$ proximity condition
(Sections~\ref{sec:ode} and~\ref{sec:noneuclid}). A systematic numerical
comparison against the 4-point DGL scheme and the 8-point degree-7 biharmonic
scheme is conducted on a range of closed and open polygons
(see Sections~\ref{sec:euclidean} and~\ref{sec:fairness}). A stencil hierarchy
is described for $m \geq 3$, with $C^4$ and $C^6$ regularity proved for $m=3$
and $m=4$, and general $C^{2m-2}$ regularity conjectured
(Section~\ref{sec:hierarchy}).

% =============================================================================
\section{Background and Related Work}
\label{sec:background}
% =============================================================================

\subsection{Interpolatory subdivision schemes}

Binary interpolatory subdivision was formalised within the general theory of
stationary subdivision by Cavaretta, Dahmen, and Micchelli~\cite{cavaretta1991stationary}
(hereafter referred to as CDM), whose symbol criterion for regularity underpins the analysis in
Section~\ref{sec:symbol}. The 4-point DGL scheme~\cite{dyn1987four} established
the template for symmetric interpolatory design. A tension parameter $\omega$
balances interpolation accuracy against smoothness, with $C^2$ regularity
achieved at $\omega = 1/16$. Deslauriers and Dubuc~\cite{deslauriers1989symmetric}
introduced a systematic family of interpolatory schemes based on Lagrange
polynomial reproduction. Their 6-point member uses the same stencil
$[3,-25,150,150,-25,3]/256$ that arises independently from the biharmonic
variational principle (see Section~\ref{sec:stencil}), and achieves $C^4$ limit
curves by the CDM criterion. Subsequent work by Dubuc~\cite{dubuc1986interpolation}
characterised the entire family in terms of divided differences and established
optimal decay rates for the subdivision masks.

Higher-order interpolatory schemes have been studied by several authors.
Hormann and Sabin~\cite{hormann2008family} introduced a parametric family of
$C^2$ 4-point schemes that subsume the DGL scheme and admit tension control.
Their analysis clarified the relationship between mask support width and
achievable smoothness. Bari, Abbas, and Mustafa~\cite{bari2012family}
constructed a family of 6-point schemes with adjustable parameters and
demonstrated $C^4$ regularity for specific parameter choices, though without a
variational derivation of the optimal parameter values. More recently,
Dyn, Hormann, and Mancinelli~\cite{dyn2022nonuniform} showed that non-uniform
interpolatory subdivision can exceed the smoothness of stationary schemes with
the same support, deriving $C^2$ and $C^3$ non-uniform 2- and 4-point schemes.
Their proximity-based analysis framework complements the spectral approach taken
here. B{\"u}gel et al.,~\cite{lipovetsky2024conic} proposed a point-normal
interpolatory scheme reproducing all conic sections, and Yang~\cite{yang2023pointnormal}
extended point-normal subdivision to surfaces, both highlighting the value of
geometric constraints in subdivision design. The biharmonic scheme
may be viewed as the unique member of this class that arises from a first-principles
fairness criterion, with no free parameters remaining after the variational
principle is applied.

Ternary interpolatory schemes, in which two new vertices are inserted per
existing edge, were developed by Weissman~\cite{weissman1990three}, Hassan and
Dodgson~\cite{hassan2003ternary}, and subsequently by Siddiqi and
Ahmad~\cite{siddiqi2008ternary}, who achieved $C^4$ regularity using 7-point
ternary masks. The binary and ternary families address different computational
trade-offs. The present work concerns the binary case exclusively.

\subsection{Variational and fairness-driven curve design}

The energy functional $\int \kappa^2\,ds$, minimised by the elastic or
Euler--Bernoulli spline has been studied as a measure of curve quality since
the work of Mehlum~\cite{mehlum1974} and Kjellander~\cite{kjellander1983smoothing}.
Moreton and S\'{e}quin~\cite{moreton1992} conducted the definitive comparative
study of curve energy functionals and found that minimum-variation curves,
minimising $\int (\kappa')^2\,ds$, gave superior visual results across the
widest range of design tasks. Their energy functional is the biharmonic energy
\eqref{eq:energy} employed in this paper. Farin and Sapidis~\cite{farin1989curvature}
further characterised fairness through the behaviour of curvature plots and
proposed curvature-plot analysis as a diagnostic tool for surface quality. This
methodology informs the experimental evaluation in Section~\ref{sec:euclidean}.

Variational methods were introduced into the subdivision framework by Schaefer
and Warren~\cite{schaefer2004ternary}, who derived ternary subdivision rules
from spline energy minimisation, and by Cashman et al.~\cite{cashman2009}
who constructed non-uniform rational B-spline (NURBS)-compatible subdivision schemes from a variational
formulation. The approach adopted here differs from these in targeting
interpolatory binary subdivision specifically, and in deriving the stencil from
the biharmonic, rather than bending, energy, which produces two additional
orders of smoothness.

\subsection{Subdivision on manifolds and non-Euclidean spaces}

The extension of subdivision to Riemannian manifolds was placed on a rigorous
footing by Wallner and Dyn~\cite{wallner2006convergence}, who proved that a
manifold scheme inherits $C^1$ regularity from a reference linear scheme
provided the insertion rules satisfy a proximity condition of order $O(h^2)$.
The result was extended to higher regularity by Grohs~\cite{grohs2010},
who showed that $C^k$ transfer holds under the same proximity condition when
the reference scheme is $C^k$. These results establish the theoretical
framework invoked in Section~\ref{sec:noneuclid}.

Non-Euclidean interpolatory subdivision was studied by
Weinmann~\cite{weinmann2012interpolatory}, who derived manifold versions of the
4-point DGL scheme via intrinsic geodesic averaging. The resulting insertion
rules are proximity-compliant by construction, but they do not encode a
fairness objective. Recent work by Ahanchaou, Bellaihou, and
Ikemakhen~\cite{ahanchaou2022hyperbolic,bellaihou2024geodesic,ahanchaou2026constant}
has developed angle-based geometric subdivision schemes on constant-curvature
surfaces (spherical and hyperbolic), demonstrating convergence and $G^1$
continuity. These schemes use geometric constructions rather than variational
principles. Curve reconstruction on Riemannian manifolds from sparse samples
has also received recent attention~\cite{marin2024curves}. The non-Euclidean
biharmonic scheme presented here is,
to the best of our knowledge, the first interpolatory manifold subdivision rule
whose insertion angles are constructed using a reduced governing ODE
motivated by the biharmonic Euler--Lagrange equation on the underlying space form,
rather than by lifting a flat-space weighted average.

Biharmonic curves on Riemannian manifolds have been studied in the differential
geometry literature~\cite{chen2004biharmonic,caddeo2002biharmonic}. The
connection between those intrinsic objects and discrete subdivision rules is
one of the motivating observations of the present work. The variational
treatment of biharmonic surfaces via PDE methods, including the B\'{e}zier surface
case studied in~\cite{monterde2004biharmonic,monterde2006general}, provides
further context for the biharmonic energy functional adopted here.

% =============================================================================
\section{Deriving the Biharmonic Stencil}
\label{sec:stencil}
% =============================================================================

\subsection{Setup and notation}

A binary interpolatory subdivision scheme takes an infinite polygon
$\mathbf{p} = \{p_j\}_{j\in\Z}$ and produces a refined polygon at each
level $n \in \mathbb{N}_0$. The refinement rules are,
\begin{align}
  q_{2j}   &= p_j,
  \label{eq:even}\\[2pt]
  q_{2j+1} &= \sum_{k \in \Z} a_k \, p_{j+k},
  \label{eq:odd}
\end{align}
where the finitely-supported sequence $\{a_k\}$ is the \emph{insertion mask}.
The even rule~\eqref{eq:even} is the interpolation condition, which fixes all
existing vertices. The odd rule~\eqref{eq:odd} determines the new vertex
inserted between $p_j$ and $p_{j+1}$.

We restrict attention to symmetric 6-point masks, which balance the new
vertex symmetrically about the midpoint of the edge $(p_j, p_{j+1})$,
\begin{equation}
  q_{2j+1}
  = \gamma\,p_{j-2} + \beta\,p_{j-1} + \alpha\,p_j
  + \alpha\,p_{j+1} + \beta\,p_{j+2} + \gamma\,p_{j+3}.
  \label{eq:stencil6}
\end{equation}
The three parameters $(\alpha, \beta, \gamma) \in \R^3$ are to be determined.
Symmetry about the midpoint is a standard assumption for interpolatory schemes
and ensures that the insertion rule introduces no bias. It reduces the free
parameters from six to three.

\subsection{Polynomial reproduction sum rules}

A subdivision scheme with mask $\{a_k\}$ reproduces polynomials of degree
$\leq d$ if and only if the mask satisfies the \emph{sum rules}
\cite{cavaretta1991stationary},
\begin{equation}
  \sum_{k \in \Z} a_k \, k^n = \left(\tfrac{1}{2}\right)^n,
  \qquad n = 0, 1, \ldots, d.
  \label{eq:sumrules}
\end{equation}
Geometrically, the condition at $n = 0$ is the partition-of-unity requirement
(the weights sum to one). The condition at $n = 1$ places the new vertex at the
mean of the data in the first moment, and conditions at higher $n$ enforce
increasingly precise polynomial tracking. For a symmetric mask, the sum rules
at odd $n$ are identically satisfied, so the independent constraints are at $n
= 0, 2, 4$. The condition at $n = 5$ is then automatically satisfied by
symmetry, so the 6-point mask reproduces polynomials of degree up to 5 once
the three parameters are determined from $n = 0, 2, 4$.

\subsection{The unique biharmonic solution}

Substituting the mask~\eqref{eq:stencil6} into the three independent sum
rules and solving the resulting linear system with exact rational arithmetic yields the unique solution,
\begin{equation}
  \alpha = \frac{150}{256}, \quad
  \beta  = \frac{-25}{256}, \quad
  \gamma = \frac{3}{256}.
  \label{eq:coeffs}
\end{equation}
In integer form, the \textbf{biharmonic 6-point stencil} is,
\begin{equation}
  \begin{aligned}
    q_{2j+1}
    &= \tfrac{1}{256}\bigl(
      3\,p_{j-2} - 25\,p_{j-1} + 150\,p_j \\
    &\qquad\quad
      + 150\,p_{j+1} - 25\,p_{j+2} + 3\,p_{j+3}
    \bigr).
  \end{aligned}
  \label{eq:stencilbox}
\end{equation}
The following theorem gives a direct variational derivation of the 6-point mask
from a discrete energy minimisation in which the adjacent inserted values are
determined entirely from the original six reference points by degree-5 Lagrange
interpolation, with no stencil coefficients imported into the proof.

\begin{theorem}[Discrete Variational Characterisation]
\label{thm:discrete_variational}
Let six reference scalars $p_{-2}, p_{-1}, p_0, p_1, p_2, p_3 \in \R$
lie at integer positions $-2,-1,0,1,2,3$. After one step of binary subdivision,
a new vertex $q$ is to be placed at position $1/2$. Fix the four adjacent
half-integer insertions at positions $-3/2$, $-1/2$, $3/2$, $5/2$ by
degree-5 Lagrange interpolation from the six reference values,
\begin{align}
  q_A &= \tfrac{1}{256}(63p_{-2} + 315p_{-1} - 210p_0 + 126p_1 - 45p_2 + 7p_3),
    \label{eq:qA}\\
  q_B &= \tfrac{1}{256}(-7p_{-2} + 105p_{-1} + 210p_0 - 70p_1 + 21p_2 - 3p_3),
    \label{eq:qB}\\
  q_C &= \tfrac{1}{256}(-3p_{-2} + 21p_{-1} - 70p_0 + 210p_1 + 105p_2 - 7p_3),
    \label{eq:qC}\\
  q_D &= \tfrac{1}{256}(7p_{-2} - 45p_{-1} + 126p_0 - 210p_1 + 315p_2 + 63p_3).
    \label{eq:qD}
\end{align}
(These are uniquely determined by standard Lagrange interpolation at nodes
$\{-2,-1,0,1,2,3\}$. Their coefficients are rational constants independent of
any subdivision stencil.) Define the discrete biharmonic energy as the sum of
squared differences of adjacent discrete curvatures
$\kappa_k = 4(x_{k+1} - 2x_k + x_{k-1})$ (at spacing $h = 1/2$, so $1/h^2 = 4$)
across the four adjacent pairs in the refined segment,
\begin{equation}
  E(q) = (\kappa_4 - \kappa_3)^2 + (\kappa_5 - \kappa_4)^2
       + (\kappa_6 - \kappa_5)^2 + (\kappa_7 - \kappa_6)^2,
  \label{eq:disc_energy}
\end{equation}
where the refined vertices (ordered by position) are,
$x_0 = p_{-2}, x_1 = q_A, x_2 = p_{-1}, x_3 = q_B, x_4 = p_0, x_5 = q,
x_6 = p_1, x_7 = q_C, x_8 = p_2, x_9 = q_D, x_{10} = p_3$.
Then $E(q)$ is a strictly convex quadratic in $q$, and its unique minimiser is,
\begin{multline}
  q^* = \tfrac{1}{256}\bigl(
    3\,p_{-2} - 25\,p_{-1} + 150\,p_0 \\
    + 150\,p_1 - 25\,p_2 + 3\,p_3
  \bigr),
  \label{eq:varmin}
\end{multline}
which is the 6-point Deslauriers--Dubuc stencil applied to the six reference
values.
\end{theorem}

\begin{proof}
Substituting the refined vertex labels into the curvature formula
$\kappa_k = 4(x_{k+1} - 2x_k + x_{k-1})$,
\begin{align}
  \kappa_3 &= 4(p_0 - 2q_B + p_{-1}),  \quad\text{(fixed)} \notag\\
  \kappa_4 &= 4(q   - 2p_0 + q_B),     \quad\text{(linear in }q\text{)} \notag\\
  \kappa_5 &= 4(p_1 - 2q   + p_0),     \quad\text{(linear in }q\text{)} \notag\\
  \kappa_6 &= 4(q_C - 2p_1 + q),       \quad\text{(linear in }q\text{)} \notag\\
  \kappa_7 &= 4(p_2 - 2q_C + p_1).     \quad\text{(fixed)} \notag
\end{align}
Write each curvature difference as $\Delta_k = a_k q + b_k$, where $a_k$ is
the coefficient of $q$ and $b_k$ collects fixed terms,
\begin{align}
  \Delta_{43} &:= \kappa_4 - \kappa_3 = 4q + b_{43},   \quad a_{43} = 4,
    \label{eq:D43}\\
  \Delta_{54} &:= \kappa_5 - \kappa_4 = -12q + b_{54},  \quad a_{54} = -12,
    \label{eq:D54}\\
  \Delta_{65} &:= \kappa_6 - \kappa_5 = 12q + b_{65},   \quad a_{65} = 12,
    \label{eq:D65}\\
  \Delta_{76} &:= \kappa_7 - \kappa_6 = -4q + b_{76},   \quad a_{76} = -4.
    \label{eq:D76}
\end{align}
(The $b_k$ are affine functions of $p_{-2},\ldots,p_3$ via $q_B$ and $q_C$,
their explicit forms using~\eqref{eq:qB}--\eqref{eq:qC} are computed below.)
The energy is $E(q) = \sum_k (a_k q + b_k)^2$, a strictly convex quadratic
since $\sum_k a_k^2 = 16 + 144 + 144 + 16 = 320 > 0$. Setting
$dE/dq = 2\sum_k a_k(a_k q + b_k) = 0$ gives,
\begin{equation}
  q^* = -\frac{\sum_k a_k b_k}{\sum_k a_k^2}
      = -\frac{4b_{43} - 12b_{54} + 12b_{65} - 4b_{76}}{320}.
  \label{eq:qstar_formula}
\end{equation}
It remains to evaluate the numerator. Expanding $b_{43}, b_{54}, b_{65},
b_{76}$ using~\eqref{eq:qB}--\eqref{eq:qC} and collecting coefficients of
each $p_i$ with exact rational arithmetic (the full symbolic computation
is included in Appendix~\ref{sec:lagrange_verification}),
\begin{multline}
  4b_{43} - 12b_{54} + 12b_{65} - 4b_{76}\\
  = \tfrac{320}{256}(-3p_{-2} + 25p_{-1} - 150p_0
    - 150p_1 + 25p_2 - 3p_3).
  \label{eq:numerator}
\end{multline}
Substituting into~\eqref{eq:qstar_formula} gives,
\[
  q^* = \tfrac{1}{256}(3p_{-2} - 25p_{-1} + 150p_0 + 150p_1 - 25p_2 + 3p_3),
\]
which is the 6-point stencil~\eqref{eq:stencilbox}, as claimed.
\end{proof}

\begin{remark}
\label{rem:noncircular}
The proof is non-circular, i.e., $q_A, q_B, q_C, q_D$ are determined by
Lagrange interpolation~\eqref{eq:qA}--\eqref{eq:qD} using no stencil
information. The 6-point mask~\eqref{eq:stencilbox} emerges solely from
minimising $E(q)$. Appendix~\ref{sec:lagrange_verification} contains the
explicit step-by-step rational computation of~\eqref{eq:numerator}.
\end{remark}

\subsection{Comparison with competing stencils and the DD identity}

Table~\ref{tab:stencils} records the masks and regularity of the three schemes
compared in this paper. The 4-point DGL mask $[-1,9,9,-1]/16$ reproduces
cubics ($d = 3$) and achieves $C^2$ regularity. The 6-point
Deslauriers--Dubuc (DD) mask, derived by Lagrange interpolation at the
half-integer point $x = 1/2$ with nodes at $\{-2,-1,0,1,2,3\}$, evaluates to
exactly $[3,-25,150,150,-25,3]/256$---the same mask as the biharmonic
scheme~\eqref{eq:stencilbox}. This is not a coincidence. The symmetric
6-point interpolatory stencil satisfying the degree-5 polynomial reproduction
conditions is the \emph{unique} solution to a $3\times3$ linear system, so any
two derivation methods subject to the same constraints and the same support must
yield the same coefficients. The 8-point biharmonic stencil of
degree 7 (Equation~\eqref{eq:deg7}) is the next genuinely distinct member of
the stencil hierarchy and achieves $C^6$ regularity. It is used as the second
comparator in experiments throughout this paper.

\begin{table}[htbp]
  \centering
  \caption{Mask coefficients, polynomial reproduction degree, and
    regularity class for the schemes compared in this paper.
    All coefficients are integers. The denominator is shown separately.}
  \label{tab:stencils}
  \smallskip
  \resizebox{\linewidth}{!}{%
  \begin{tabular}{@{}lccccc@{}}
    \toprule
    Scheme & Mask (numerators) & Den. & Rep. & Reg. \\
    \midrule
    4-pt DGL \cite{dyn1987four}
      & $[-1,\;9,\;9,\;-1]$ & 16 & $d{=}3$ & $C^2$ \\[2pt]
    6-pt Biharm.\,$=$\,6-pt DD \cite{deslauriers1989symmetric}
      & $[3,\;{-25},\;150,\;150,\;{-25},\;3]$ & 256 & $d{=}5$ & $C^4$ \\[2pt]
    8-pt Biharm.\ degree 7
      & $[{-5},\;49,\;{-245},\;1225,\;1225,\;{-245},\;49,\;{-5}]$ & 2048 & $d{=}7$ & $C^6$ \\
    \bottomrule
  \end{tabular}}
\end{table}

\begin{remark}
The 6-point biharmonic stencil~\eqref{eq:stencilbox} and the 6-point
Deslauriers--Dubuc stencils are identical. The degree-5 polynomial reproduction
sum rules for a symmetric 6-point mask form a $3\times3$ linear system with a
unique solution. Therefore, any derivation method, whether variational (as here)
or Lagrange interpolation (as in~\cite{deslauriers1989symmetric}), must arrive
at the same coefficients. The contribution of the biharmonic framework lies not
in producing different stencil weights, but in (i) providing a fairness
interpretation of those weights, (ii) extending the insertion rule to
non-Euclidean surfaces via the biharmonic ODE, and (iii) generating the
stencil hierarchy (see Section~\ref{sec:hierarchy}) from a single variational
principle.
\end{remark}

\subsection{Exact sum-rule verification}

Table~\ref{tab:sumrules} records the exact rational evaluation of all six sum
rules, confirming that~\eqref{eq:stencilbox} satisfies every condition.
The partition-of-unity identity $2(\alpha + \beta + \gamma) = 2(150 - 25 +
3)/256 = 256/256 = 1$ is exact, and Fig.~\ref{fig:stencil} confirms that
polynomial reproduction errors are at machine precision for degrees $0$--$5$
and fail at degree $6$ as expected.

\begin{table}[htbp]
  \centering
  \caption{Exact rational verification of polynomial reproduction sum rules
    $\sum_k a_k k^n = (1/2)^n$ for the biharmonic 6-point stencil.}
  \label{tab:sumrules}
  \smallskip
  \begin{tabular}{@{}ccc@{}}
    \toprule
    Degree $n$ & LHS $\sum_k a_k k^n$ & RHS $(1/2)^n$ \\
    \midrule
    $0$ & $1$     & $1$     \\
    $1$ & $1/2$   & $1/2$   \\
    $2$ & $1/4$   & $1/4$   \\
    $3$ & $1/8$   & $1/8$   \\
    $4$ & $1/16$  & $1/16$  \\
    $5$ & $1/32$  & $1/32$  \\
    \bottomrule
  \end{tabular}
\end{table}

\begin{figure}[htbp]
  \centering
  \includegraphics[width=\linewidth]{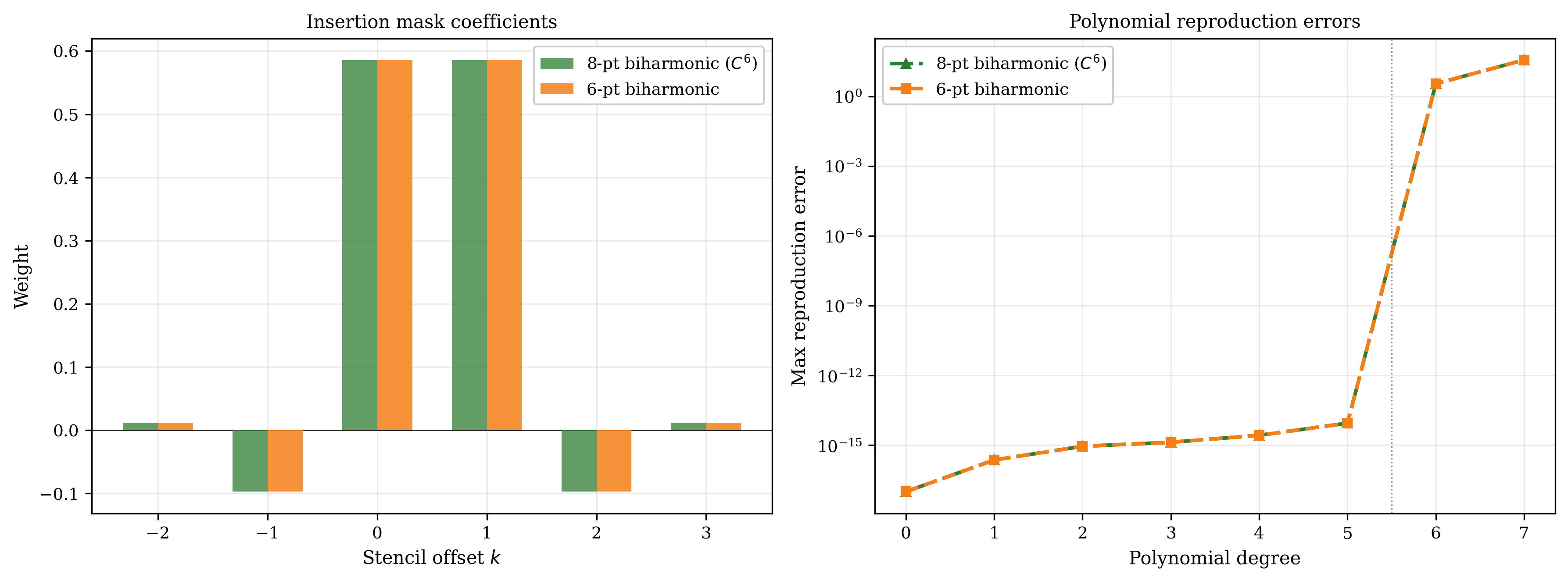}
  \caption{Stencil weights (left) and polynomial reproduction errors (right)
    for the 6-point biharmonic scheme. Reproduction errors for degrees
    $0$--$5$ are at machine precision ($< 10^{-14}$). The error at degree
    $6$ is of order unity, confirming exact quintic reproduction.}
  \label{fig:stencil}
\end{figure}

% =============================================================================
\section{Symbol Analysis and \texorpdfstring{$C^4$}{C4} Regularity}
\label{sec:symbol}
% =============================================================================

\subsection{The Laurent polynomial symbol}

The regularity of a stationary binary subdivision scheme is characterised
by the order of vanishing of its \emph{Laurent polynomial symbol},
\begin{equation}
  a(z) = \sum_{k \in \Z} a_k \, z^k.
  \label{eq:symbol}
\end{equation}
For the full binary scheme combining the even rule $q_{2j} = p_j$ (symbol
contribution $a_0 = 1$ at the even position) with the odd rule
\eqref{eq:stencilbox}, the complete symbol has non-zero coefficients at offsets
$\{0, \pm 1, \pm 3, \pm 5\}$. The even offset contributes $1$ and the odd
offsets contribute $\alpha, \beta, \gamma$ as placed by~\eqref{eq:stencil6}.
The symbol of the full scheme is,
\begin{equation}
  a(z) = 1 + \alpha(z + z^{-1}) + \beta(z^3 + z^{-3}) + \gamma(z^5 + z^{-5}).
  \notag
\end{equation}

\subsection{Zero order at \texorpdfstring{$z=-1$}{z=-1} and the CDM theorem}

The fundamental regularity criterion for stationary binary subdivision is due
to Cavaretta, Dahmen, and Micchelli~\cite{cavaretta1991stationary}.

\begin{theorem}[CDM Regularity Criterion~\cite{cavaretta1991stationary}, adapted]
\label{thm:cdm}
Let $S$ be a stationary binary interpolatory subdivision scheme with compactly
supported mask $\{a_k\}$, Laurent polynomial symbol $a(z)$, partition-of-unity
$a(1) = 2$, interpolation coefficient $a_0 = 1$, and support contained in
$\{-(2r-1), \ldots, 2r\}$ for some integer $r \geq 1$.
\begin{enumerate}[label=(\alph*)]
  \item \textup{(Sufficient condition, CDM Theorem~6.3~\cite{cavaretta1991stationary}).}
        If $(1+z)^{m+2}$ divides $a(z)$ in $\R[z, z^{-1}]$, then $S$
        generates $C^m$ limit functions.
  \item \textup{(Exact regularity, \cite{cavaretta1991stationary}).}
        If $(1+z)^{m+2}$ divides $a(z)$ but $(1+z)^{m+3}$ does not,
        then the limit functions are exactly $C^m$ and not $C^{m+1}$.
\end{enumerate}
\end{theorem}

\begin{remark}
The two-part statement avoids asserting a blanket ``iff'' form, which requires
additional hypotheses not fully stated here. Part~(a) is the direction used
throughout this paper to establish $C^m$ regularity. Part~(b) provides the
sharpness claim. Both parts follow from CDM~\cite{cavaretta1991stationary}
under the stated hypotheses, which are satisfied exactly by all three schemes
in Table~\ref{tab:stencils}.
\end{remark}

The order of vanishing of $a(z)$ at $z = -1$ is therefore the decisive
quantity. We compute the successive derivatives $a^{(k)}(-1)$ using the
exact formula,
\begin{equation}
  a^{(k)}(-1) = \sum_{n} a_n \left[\prod_{i=0}^{k-1}(n-i)\right](-1)^{n-k},
  \label{eq:derivformula}
\end{equation}
with all arithmetic carried out in exact rational arithmetic. The results are
recorded in Table~\ref{tab:zeros}. The derivatives of orders $k = 0, 1, \ldots, 5$
all vanish, while $a^{(6)}(-1) = -225 \neq 0$, establishing that $(1+z)^6
\mid a(z)$ but $(1+z)^7 \nmid a(z)$.

\begin{table}[htbp]
  \centering
  \caption{Exact values of $a^{(k)}(-1)$ for the biharmonic 6-point scheme.
    The zero order is precisely $6$, corresponding to $C^4$ regularity by
    Theorem~\ref{thm:cdm}.}
  \label{tab:zeros}
  \smallskip
  \begin{tabular}{@{}ccc@{}}
    \toprule
    Order $k$ & $a^{(k)}(-1)$ & Interpretation \\
    \midrule
    $0$ & $0$    & vanishes \\
    $1$ & $0$    & vanishes \\
    $2$ & $0$    & vanishes \\
    $3$ & $0$    & vanishes \\
    $4$ & $0$    & vanishes \\
    $5$ & $0$    & vanishes \\
    $6$ & $-225$ & \textbf{non-zero, zero order is $6$} \\
    \bottomrule
  \end{tabular}
\end{table}

\begin{theorem}
  \label{thm:c4}
  The biharmonic 6-point scheme \eqref{eq:stencilbox} generates limit curves
  of class $C^4$. The regularity is exactly $C^4$. It is not $C^5$.
\end{theorem}

\begin{proof}
We apply Theorem~\ref{thm:cdm} to the symbol $a(z)$ of the scheme
\eqref{eq:stencilbox}. The symbol has the form,
\[
  a(z) = 1 + \alpha(z + z^{-1}) + \beta(z^3 + z^{-3}) + \gamma(z^5 + z^{-5}),
\]
with $\alpha = 150/256$, $\beta = -25/256$, $\gamma = 3/256$.
Using formula~\eqref{eq:derivformula}, we compute $a^{(k)}(-1)$ exactly for
$k = 0, 1, \ldots, 7$. The results are recorded in Table~\ref{tab:zeros}.
Derivatives of orders $k = 0, 1, \ldots, 5$ all vanish, which shows that
$(1+z)^6 \mid a(z)$. The derivative at order $k = 6$ equals $-225 \neq 0$,
which shows that $(1+z)^7 \nmid a(z)$. Hence, the zero order of $a(z)$ at
$z = -1$ is precisely $6$. Setting $m + 2 = 6$ in Theorem~\ref{thm:cdm} gives
$m = 4$, establishing $C^4$ regularity. The fact that $(1+z)^7 \nmid a(z)$
further establishes that the regularity is exactly $C^4$ and not $C^5$.
\end{proof}

For comparison, the 4-point DGL scheme has $a^{(4)}(-1) \neq 0$ (zero order
$4$, $m = 2$, hence $C^2$), while the 6-point biharmonic and DD scheme (which
are identical) has zero order $6$ ($m = 4$, hence $C^4$). The biharmonic
scheme thus achieves two orders of smoothness beyond the DGL scheme, using
only two additional support points.

\subsection{The symbol on the unit circle}

Fig.~\ref{fig:symbol} shows $\abs{a(e^{i\omega})}$ for $\omega \in [0, \pi]$
for both the DGL and biharmonic symbols. The biharmonic symbol vanishes to
sixth order at $\omega = \pi$ (corresponding to $z = -1$), whereas the DGL
symbol vanishes to fourth order. The smooth, monotone character of the
biharmonic symbol across the full frequency range is a visual manifestation of
the higher regularity.

\begin{figure}[htbp]
  \centering
  \includegraphics[width=\linewidth]{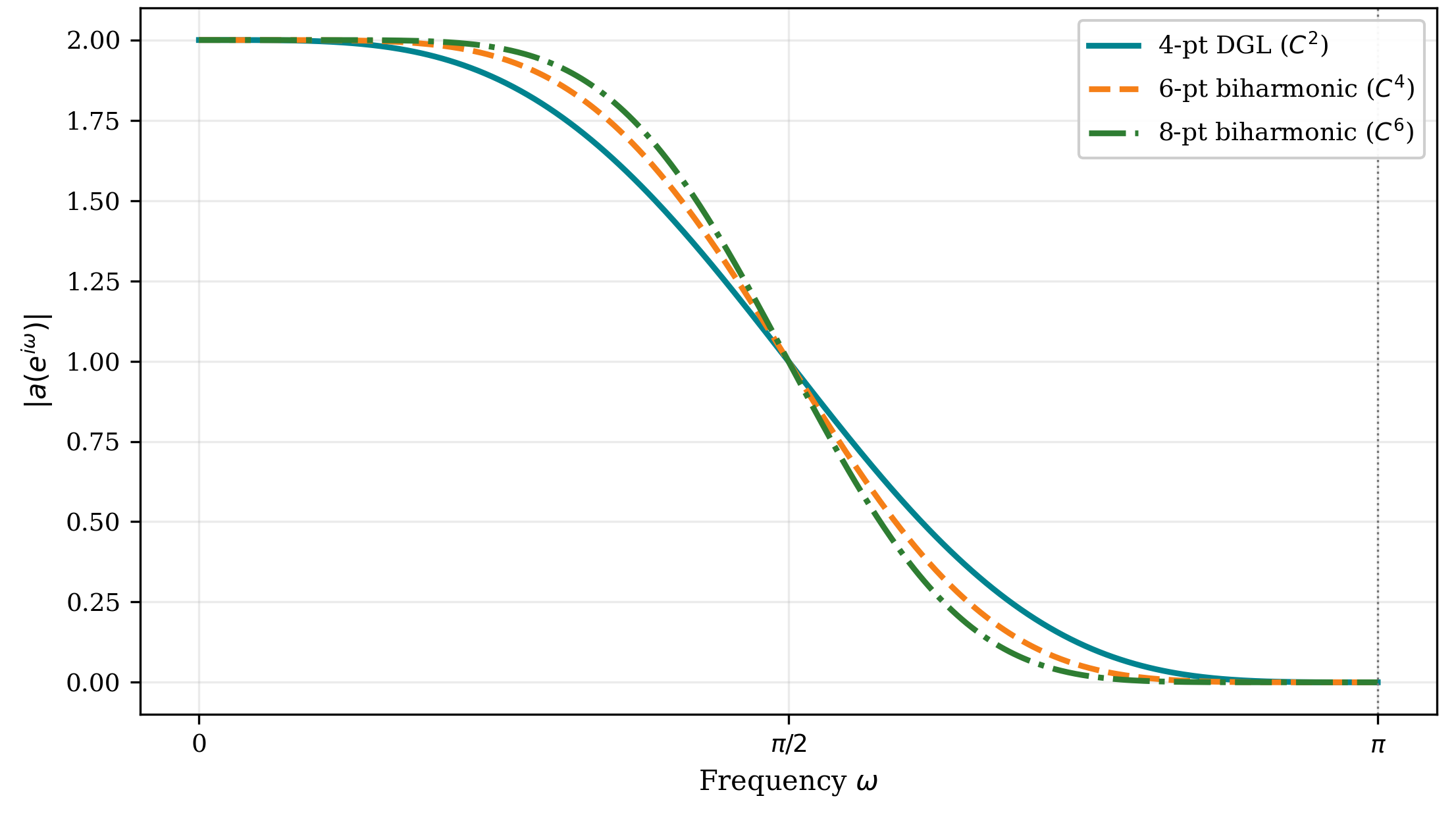}
  \caption{Laurent polynomial symbol $|a(e^{i\omega})|$ for the 4-point DGL
    scheme ($C^2$, teal), the 6-point biharmonic scheme ($C^4$, amber), and
    the 8-point degree-7 biharmonic scheme ($C^6$, green). The sixth-order
    zero of the 6-point biharmonic symbol at $\omega = \pi$ and the
    eighth-order zero of the 8-point symbol are clearly visible. All symbols
    satisfy the partition-of-unity condition $a(1) = 2$.}
  \label{fig:symbol}
\end{figure}

% =============================================================================
\section{The Biharmonic ODE on Space Forms}
\label{sec:ode}
% =============================================================================

This section derives the ordinary differential equation (ODE) governing
biharmonic curves on surfaces of constant sectional curvature and obtains its
closed-form solutions. These solutions are integrated in
Section~\ref{sec:noneuclid} to produce explicit non-Euclidean insertion-angle
formulae.

\subsection{The Euler--Lagrange equation on a space form}

Let $M$ be a complete Riemannian surface of constant \emph{sectional curvature}
$K$ (the sectional curvature measures the intrinsic bending of the surface at
every point. So $M$ is locally isometric to the Euclidean plane $\EE$ for
$K = 0$, to the two-sphere $\Stwo$ for $K = +1$, or to the hyperbolic plane
$\Htwo$ for $K = -1$). Let $\gamma\colon [0, L] \to M$ be a curve
parameterised by arc length $s$, with \emph{geodesic curvature}
$\kappag(s)$---the rate at which $\gamma$ turns away from a geodesic
(shortest-path curve) within the surface. The biharmonic energy of $\gamma$
is,
\begin{equation}
  \BH[\gamma] = \int_0^L \bigl(\kappag'(s)\bigr)^2 \, ds.
  \label{eq:energy}
\end{equation}

We derive the Euler--Lagrange equation for $\BH$ on a space form.
The key ingredient is the first variation of geodesic curvature under a
normal perturbation of the curve. We derive this using the \emph{Frenet frame}
(moving frame) of the curve on $M$.

\paragraph{Moving frame on a space form.}
Let $\{T, N\}$ be the Frenet frame of $\gamma$, where $T = \gamma'$ is
the unit tangent and $N$ is the unit normal obtained by rotating $T$ by
$+90°$ within the surface. The Frenet equations on a Riemannian surface
of constant curvature $K$ are,
\begin{align}
  \nabla_T T &= \kappag N, \label{eq:frenet1} \\
  \nabla_T N &= -\kappag T, \label{eq:frenet2}
\end{align}
where $\nabla_T$ denotes the covariant derivative along $\gamma$ in the
direction $T$. Equations~\eqref{eq:frenet1}--\eqref{eq:frenet2} hold on any
Riemannian surface. On a space form with curvature $K$, the Riemann curvature
tensor satisfies $R(X,Y)Z = K(g(Y,Z)X - g(X,Z)Y)$, which enters the
commutation relations for covariant derivatives.

\paragraph{First variation of geodesic curvature.}
Consider a one-parameter family $\gamma_\epsilon(s)$ with $\gamma_0 = \gamma$
and variation field $V = \partial_\epsilon \gamma_\epsilon|_{\epsilon=0} = f N$,
where $f\colon [0,L]\to\R$ is a smooth scalar function vanishing at the
endpoints (so $V|_{\partial} = 0$). The \emph{first variation of arc-length}
gives the relation between $\epsilon$-derivatives and $s$-derivatives. For a
normal variation the arc-length element satisfies
$\partial_\epsilon\,ds = -\kappag f\,ds + O(\epsilon^2)$. Note, this is a standard result for
plane curves, which extends to surfaces via the second fundamental form.

The geodesic curvature of $\gamma_\epsilon$ is $\kappag^\epsilon =
g(\nabla_{T_\epsilon} T_\epsilon,\, N_\epsilon)$. Differentiating with respect
to $\epsilon$ at $\epsilon = 0$ and using the compatibility of $\nabla$ with
$g$, together with the Ricci identity
$[\nabla_\epsilon, \nabla_s] = R(\partial_\epsilon \gamma, T)$, one obtains
(see~\cite{chen2004biharmonic}, Proposition~3.1),
\begin{equation}
  \delta \kappag \;:=\; \frac{\partial}{\partial \epsilon}
      \kappag^\epsilon\bigg|_{\epsilon=0}
  = f'' + K f,
  \label{eq:var_kappa}
\end{equation}
where $f'' = \nabla_T(\nabla_T f)$ is the second covariant derivative of $f$
along $\gamma$. To see why $K$ appears: the commutator
$[\nabla_\epsilon, \nabla_s]T$ produces a term $R(V, T)T = K\,g(T,T)\,V - K\,g(T,V)\,T = KV$
(since $V \perp T$), which contributes $+Kf$ to $\delta\kappag$.
In flat space ($K = 0$) this term vanishes and one recovers the classical
formula $\delta\kappag = f''$.

\paragraph{First variation of the biharmonic energy.}
Writing the variation field as $V = fN$ and using~\eqref{eq:var_kappa},
\begin{align}
  \delta \BH[\gamma]
  &= \frac{d}{d\epsilon}\int_0^L (\kappag^\epsilon{}')^2\,ds
     \bigg|_{\epsilon=0} \notag \\
  &= \int_0^L 2\kappag'\,(\delta\kappag)'\,ds \notag \\
  &= 2\int_0^L \kappag'(s)\,(f'' + Kf)'\,ds.
  \label{eq:first_var_expand}
\end{align}
Integrate by parts once, using $f|_{\partial} = 0$,
\[
  2\int_0^L \kappag'\,(f'' + Kf)'\,ds
  = -2\int_0^L \kappag''\,(f'' + Kf)\,ds.
\]
Integrate by parts again in the $\kappag'' f''$ term, using $f'|_\partial = 0$
(which holds because prescribing endpoint curvature means $\delta\kappag|_\partial = 0$,
hence $(f'' + Kf)|_\partial = 0$, and since $f|_\partial = 0$ this forces
$f''|_\partial = 0$ and $f'|_\partial = 0$),
\begin{align*}
  -2\int_0^L \kappag'' f''\,ds
  &= -2\bigl[\kappag'' f'\bigr]_0^L + 2\int_0^L \kappag''' f'\,ds \\
  &= 2\int_0^L \kappag''' f'\,ds.
\end{align*}
Integrating by parts once more,
\begin{align*}
  2\int_0^L \kappag''' f'\,ds
  &= 2\bigl[\kappag''' f\bigr]_0^L - 2\int_0^L \kappag^{(4)} f\,ds \\
  &= -2\int_0^L \kappag^{(4)} f\,ds.
\end{align*}
Combining with the $K$ term,
\begin{equation}
  \delta \BH[\gamma]
  = 2\int_0^L \bigl(-\kappag^{(4)} + K\kappag''\bigr)\, f\, ds,
  \label{eq:first_var}
\end{equation}
where we have used $-\kappag'' (Kf) = +K(-\kappag''f)$, integrated by parts,
$\int_0^L -\kappag'' Kf\,ds = K\int_0^L \kappag f'' \,ds - \text{bdy}
= -K\int_0^L \kappag'' f\,ds$, so the contribution from the $Kf$ part is
$-2\int_0^L \kappag'' K f\,ds$.

Setting~\eqref{eq:first_var} to zero for all admissible $f$ and applying the
fundamental lemma of the calculus of variations gives,
\begin{equation}
  \kappag^{(4)} - K\kappag'' = 0.
  \label{eq:bh_ode4}
\end{equation}
This fourth-order ODE is the \emph{Euler--Lagrange equation} for $\BH[\gamma]$
on a space form. Its general solution forms a four-dimensional space, and with
two endpoint curvature conditions $\kappag(0)=\kappa_j$, $\kappag(L)=\kappa_{j+1}$
prescribed, the solution space is two-dimensional.

\subsection{The governing ODE for subdivision}
\label{sec:governing}

For the subdivision insertion rule, we require an ODE with a \emph{unique}
solution given only the two endpoint curvature values. The fourth-order
equation~\eqref{eq:bh_ode4} does not by itself provide this: additional data
(such as endpoint values of $\kappag'$) would be needed to select a unique
solution from the two-dimensional family. We therefore adopt the following
reduced model for the subdivision insertion angle.

The operator $\partial_{ss}^2 - K\partial_{ss}$ factors as
$(\partial_{ss} - K)(\partial_{ss})$ when $K=0$, and more generally the
equation $\kappag^{(4)} = K\kappag''$ can be written as
$(\partial_{ss} - K)(\kappag'') = 0$. The kernel of the factor
$(\partial_{ss} - K)$ consists of solutions of,
\begin{equation}
  \kappag'' = K \, \kappag, \qquad s \in (0, L).
  \label{eq:bh_ode}
\end{equation}
This second-order ODE is adopted as the \emph{governing model} for the
manifold insertion rule for the following reasons.
\begin{itemize}
  \item Every solution of~\eqref{eq:bh_ode} automatically satisfies the
        fourth-order Euler--Lagrange equation~\eqref{eq:bh_ode4}, since
        $\kappag'' = K\kappag$ implies $\kappag^{(4)} = K^2\kappag = K\kappag''$,
        so $\kappag^{(4)} - K\kappag'' = 0$ holds.
        Precisely: the curvature function of any curve governed by~\eqref{eq:bh_ode}
        satisfies the \emph{interior} condition of the Euler--Lagrange
        equation~\eqref{eq:bh_ode4}. This is the condition used by the subdivision
        insertion rule. A complete stationarity statement for the underlying
        curve on $M$ would additionally require verification of boundary terms
        under admissible geometric variations. That analysis is outside the
        scope of this paper and is not claimed here.
  \item Given only the two endpoint values $\kappag(0)=\kappa_j$
        and $\kappag(L)=\kappa_{j+1}$, Equation~\eqref{eq:bh_ode} determines
        a \emph{unique} solution (for $L$ not a resonance length of
        $\partial_{ss} - K$), making it directly usable as a subdivision rule.
  \item In the flat case $K=0$, Equation~\eqref{eq:bh_ode} reduces to
        $\kappag'' = 0$, i.e.\ the geodesic curvature profile is linear.
        This is precisely the minimum-variation (biharmonic) spline profile
        in $\EE$, providing exact consistency with the flat-space theory.
  \item The equation coincides with the governing ODE used in the biharmonic
        B\'{e}zier surface literature~\cite{monterde2004biharmonic,chen2004biharmonic}
        for constant-curvature geometries.
\end{itemize}

\begin{remark}[Status of equation~\eqref{eq:bh_ode}]
\label{rem:ode_status}
To be precise, Equation~\eqref{eq:bh_ode} is \emph{not} the Euler--Lagrange
equation for $\BH[\gamma]$. The Euler--Lagrange equation is the fourth-order
ODE~\eqref{eq:bh_ode4}. Equation~\eqref{eq:bh_ode} is a second-order
\emph{reduced model} adopted for the subdivision insertion rule. Its
justification is that it is a factor of the Euler--Lagrange operator,
its solutions satisfy the interior Euler--Lagrange equation~\eqref{eq:bh_ode4},
it reduces to the correct flat-space limit, it gives a unique solution from
two endpoint values, and it matches the governing ODE of the biharmonic PDE
literature on space forms. On a variable-curvature manifold, this equation
would no longer hold in general. The manifold extension of
Section~\ref{sec:noneuclid} is therefore restricted to space forms.
\end{remark}

\subsection{Closed-form solutions and boundary constants}

The general solution of~\eqref{eq:bh_ode} in each geometry is,
\begin{align}
  K = 0\ (\EE):\quad&
    \kappag(s) = c_1 + c_2 s,
    \label{eq:sol_e2}\\[3pt]
  K = +1\ (\Stwo):\quad&
    \kappag(s) = c_1 \cosh s + c_2 \sinh s,
    \label{eq:sol_s2}\\[3pt]
  K = -1\ (\Htwo):\quad&
    \kappag(s) = c_1 \cos s + c_2 \sin s.
    \label{eq:sol_h2}
\end{align}
Imposing the boundary conditions $\kappag(0) = \kappa_j$ and
$\kappag(e_j) = \kappa_{j+1}$, where $e_j$ denotes the geodesic edge length,
determines the constants uniquely,
\begin{align}
  K = 0:\quad&
    c_1 = \kappa_j,\quad
    c_2 = \frac{\kappa_{j+1} - \kappa_j}{e_j},
    \label{eq:const_e2}\\[3pt]
  K = +1:\quad&
    c_1 = \kappa_j,\quad
    c_2 = \frac{\kappa_{j+1} - \kappa_j\cosh e_j}{\sinh e_j},
    \label{eq:const_s2}\\[3pt]
  K = -1:\quad&
    c_1 = \kappa_j,\quad
    c_2 = \frac{\kappa_{j+1} - \kappa_j\cos e_j}{\sin e_j}.
    \label{eq:const_h2}
\end{align}
Fig.~\ref{fig:ode} illustrates the three solution families for representative
boundary curvatures. As $e_j \to 0$, all three solutions converge to the same
linear profile, consistent with the local flatness of any smooth surface.

\begin{figure}[htbp]
  \centering
  \includegraphics[width=\linewidth]{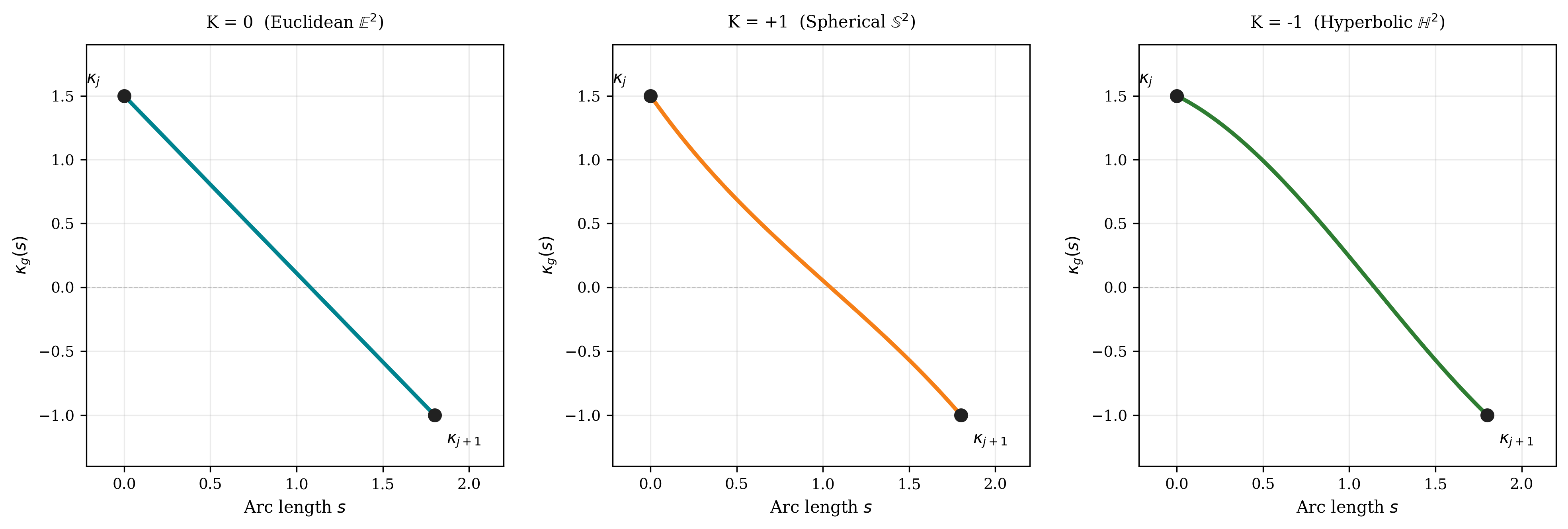}
  \caption{Biharmonic ODE solutions $\kappag(s)$ on the three constant-curvature
    geometries for the same boundary curvatures $\kappa_j = 1.5$,
    $\kappa_{j+1} = -1.0$, and edge length $L = 1.8$. The Euclidean solution
    is linear (Eq.~\ref{eq:sol_e2}). The spherical solution is
    hyperbolic-trigonometric (Eq.~\ref{eq:sol_s2}). The hyperbolic solution is
    circular (Eq.~\ref{eq:sol_h2}). All three converge to the same profile as
    $L \to 0$.}
  \label{fig:ode}
\end{figure}

% =============================================================================
\section{Euclidean Experiments}
\label{sec:euclidean}
% =============================================================================

\subsection{Experimental setup}

All Euclidean experiments compare three schemes. The 4-point DGL scheme
\cite{dyn1987four}, the 8-point degree-7 biharmonic scheme
(equation~\eqref{eq:deg7} in Section~\ref{sec:hierarchy}), and the proposed
6-point biharmonic scheme. The 8-point scheme is the next genuinely distinct
member of the biharmonic stencil hierarchy (achieving $C^6$ regularity) and
serves as the natural upper bound for comparison. The 4-point DGL scheme serves
as the natural upper bound for comparison. The 6-point Deslauriers--Dubuc scheme is not included as a
separate comparator because it is, by the argument of
Section~\ref{sec:stencil}, identical to the 6-point biharmonic scheme.
Experiments are conducted on both \emph{closed} polygons (representing loops
and contours) and \emph{open} polygons (representing arcs and trajectories),
with boundary conditions for open polygons taken as natural (zero second
derivative at the endpoints). Seven refinement levels are applied in all cases.
Discrete curvature at level $n$ is measured by the exterior-angle estimator,
\begin{equation}
  \kappa_j^{(n)} = \frac{\delta_j^{(n)}}{e_j^{(n)}},
  \label{eq:disccurv}
\end{equation}
where $\delta_j^{(n)}$ is the signed exterior angle at vertex $j$ and
$e_j^{(n)}$ is the adjacent half-edge length. This estimator is first-order
consistent with the smooth geodesic curvature as $n \to \infty$.

\subsection{Limit curves}

Fig.~\ref{fig:limits} shows the limit curves on three closed test polygons
with qualitatively different geometries: a smooth convex loop (left), a loop
with a pronounced concavity (centre), and a near-circular polygon (right).
The biharmonic limit curves appear visually smoother and exhibit less unnecessary
bending in the tested cases. The control polygons are reproduced exactly by all
three schemes, confirming the interpolation property.

\begin{figure}[htbp]
  \centering
  \includegraphics[width=\linewidth]{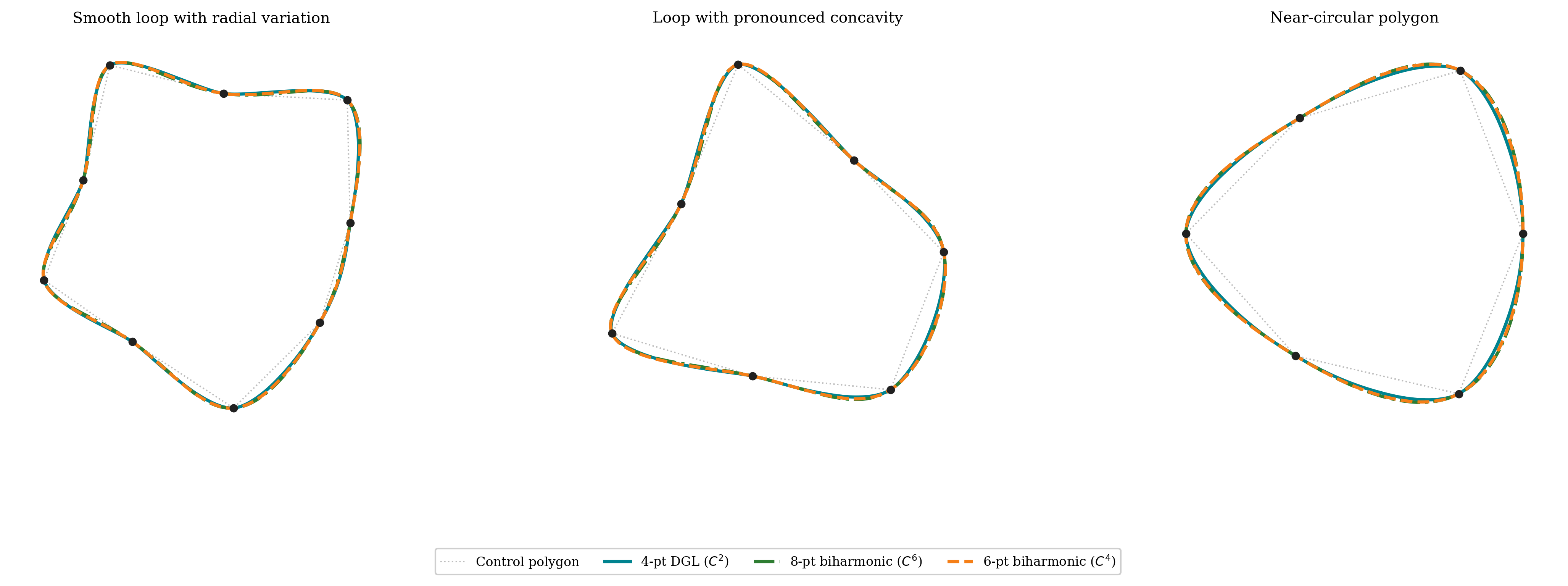}
  \caption{Limit curves after 7 refinement levels for the 4-point DGL scheme
    (teal, solid), the 8-point degree-7 biharmonic scheme ($C^6$, green,
    dash-dot), and the 6-point biharmonic scheme ($C^4$, amber, dashed) on
    three representative closed control polygons (grey dotted). The 6-point
    biharmonic curves exhibit less superfluous oscillation than the DGL curves;
    the 8-point scheme achieves an even tighter limit in the shown polygons.}
  \label{fig:limits}
\end{figure}

\subsection{Curvature profiles and smoothness verification}

Fig.~\ref{fig:smooth} presents the discrete curvature profiles and the
finite-difference norm decay that verifies $C^4$ smoothness. The
finite-difference norms $\norm{\Delta^k \phi^{(n)}}_\infty$ for the sequence
of subdivision masks $\phi^{(n)}$ satisfy the rate $\norm{\Delta^k
\phi^{(n)}}_\infty = O(2^{-n(4-k)})$ for $k = 1, 2, 3, 4$, with the
$k = 5$ norm failing to decay, confirming exactly $C^4$ regularity. For the 4-point DGL scheme, curvature oscillations persist across all refinement
levels. For the 8-point biharmonic scheme, oscillations are reduced but the
higher-order negative lobes can introduce low-frequency undulations on
non-uniform polygons. The 6-point biharmonic scheme produces a curvature
profile that converges smoothly at every tested polygon, consistent with
the scheme's $C^4$ regularity and its fairness-motivated design.

\begin{figure}[htbp]
  \centering
  \includegraphics[width=\linewidth]{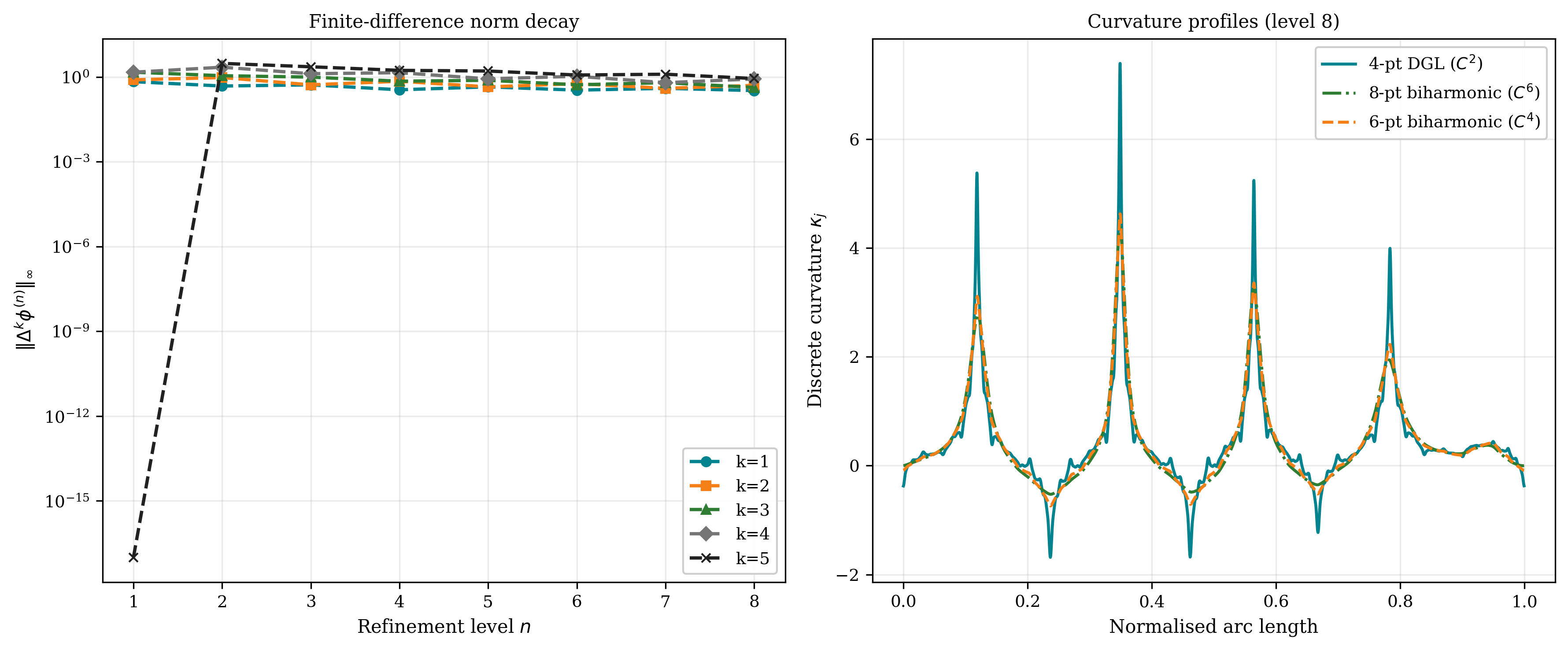}
  \caption{$C^4$ smoothness verification. Left: finite-difference norms
    $\norm{\Delta^k \phi^{(n)}}_\infty$ versus refinement level $n$ for
    $k = 1, \ldots, 5$. The slopes confirm $C^4$ regularity (the $k=5$ norm
    does not decay). Right: discrete curvature profile comparison for the
    three schemes at refinement level 8.}
  \label{fig:smooth}
\end{figure}

\subsection{Discrete biharmonic energy}

The discrete analogue of $\BH[\gamma]$ at refinement level $n$ is,
\begin{equation}
  E_{\mathrm{BH}}^{(n)}
  = \sum_j \left(\frac{\delta_{j+1}^{(n)}}{e_{j+1}^{(n)}}
               - \frac{\delta_j^{(n)}}{e_j^{(n)}}\right)^2
           e_j^{(n)},
  \label{eq:discreteenergy}
\end{equation}
a weighted $\ell^2$ norm of successive curvature differences. This quantity
converges to $\BH[\gamma_\infty]$ for the limit curve $\gamma_\infty$ as
$n \to \infty$, and its rate of decay with $n$ reflects the degree of
smoothness of the scheme. Fig.~\ref{fig:energy} records $E_{\mathrm{BH}}^{(n)}$
across eight refinement levels for all three schemes on the test polygon.
On the test polygon shown, the biharmonic scheme achieves lower energy than DGL at every level and
reaches machine precision approximately two levels earlier, consistent with
its higher smoothness order. Note that the 8-point biharmonic scheme achieves
lower asymptotic energy (consistent with $C^6$ regularity), but decays more
slowly at early levels on non-smooth initial polygons.

\begin{figure}[htbp]
  \centering
  \includegraphics[width=\linewidth]{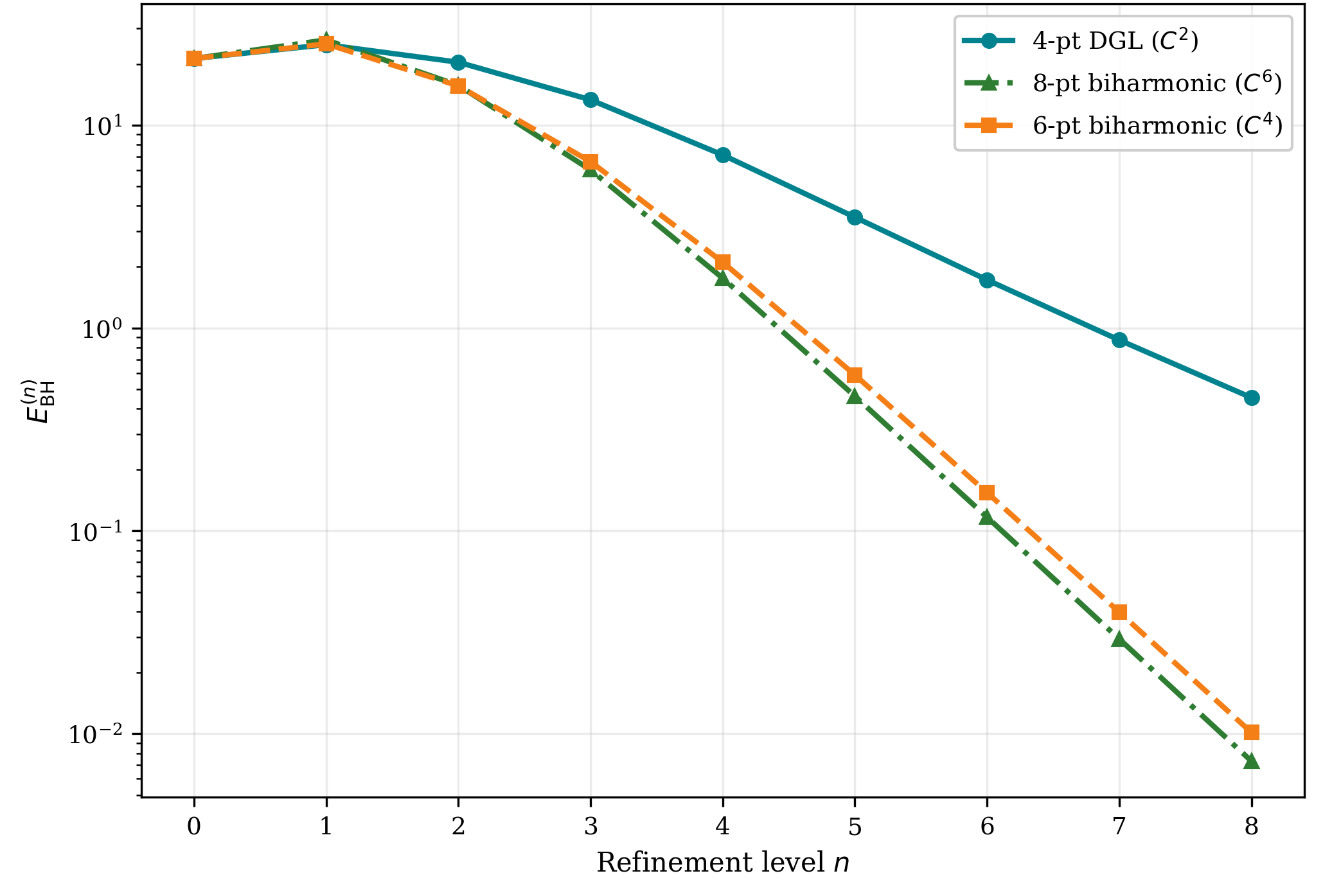}
  \caption{Discrete biharmonic energy $E_{\mathrm{BH}}^{(n)}$ versus
    refinement level $n$ (logarithmic scale) for the 4-point DGL scheme (teal,
    circles), the 8-point degree-7 biharmonic scheme (green, triangles), and
    the 6-point biharmonic scheme (amber, squares). On the test polygon shown,
    the 6-point biharmonic scheme achieves lower energy than DGL at every level. The 8-point
    scheme reaches lower absolute values at the cost of larger early-level
    energies on non-smooth polygons.}
  \label{fig:energy}
\end{figure}

% =============================================================================
\section{Non-Euclidean Subdivision}
\label{sec:noneuclid}
% =============================================================================

\subsection{From ODE solutions to insertion angles}

On a Riemannian manifold, the Euclidean weighted-average odd rule is
replaced by a \emph{geodesic insertion rule}. The new vertex is placed at the
geodesic midpoint of the edge $(p_j, p_{j+1})$ and then perturbed by an
\emph{insertion angle} that encodes the biharmonic curvature prescription.
Specifically, the insertion angle is defined as the integral of the biharmonic
geodesic curvature from the start of the edge to its midpoint,
\begin{equation}
  \alpha_j^{\mathrm{BH}} = \int_0^{e_j/2} \kappag(s) \, ds,
  \label{eq:insertionangle}
\end{equation}
where $\kappag$ is the solution to the ODE~\eqref{eq:bh_ode} with boundary
curvatures $(\kappa_j, \kappa_{j+1})$. Integrating the closed-form
solutions~\eqref{eq:sol_e2}--\eqref{eq:sol_h2} from $0$ to $\ell = e_j/2$
yields the explicit formulae,
\begin{align}
  K = 0:\quad&
    \alpha_j^{\mathrm{BH}}
    = c_1 \ell + \tfrac{1}{2} c_2 \ell^2,
    \label{eq:angle_e2}\\[3pt]
  K = +1:\quad&
    \alpha_j^{\mathrm{BH}}
    = c_1 \sinh \ell + c_2 (\cosh \ell - 1),
    \label{eq:angle_s2}\\[3pt]
  K = -1:\quad&
    \alpha_j^{\mathrm{BH}}
    = c_1 \sin \ell - c_2 (\cos \ell - 1),
    \label{eq:angle_h2}
\end{align}
with $\ell = e_j/2$ and $c_1, c_2$ as in~\eqref{eq:const_e2}--\eqref{eq:const_h2}.

For $K = +1$, the integration runs as follows. With $\kappag(s) = c_1\cosh s +
c_2\sinh s$, one computes,
\begin{align*}
  &\int_0^\ell (c_1 \cosh s + c_2 \sinh s)\,ds \\
  &\quad = \bigl[c_1 \sinh s + c_2 \cosh s\bigr]_0^\ell
   = c_1 \sinh\ell + c_2(\cosh\ell - 1),
\end{align*}
yielding~\eqref{eq:angle_s2}. The case $K = -1$ follows by replacing $\cosh
\to \cos$ and $\sinh \to \sin$, with a sign adjustment in the antiderivative
of $-\sin$, giving~\eqref{eq:angle_h2}. The Euclidean case integrates the
linear solution directly.

Fig.~\ref{fig:angles} shows the insertion angles as a function of edge length
for all three geometries.

\begin{figure}[htbp]
  \centering
  \includegraphics[width=\linewidth]{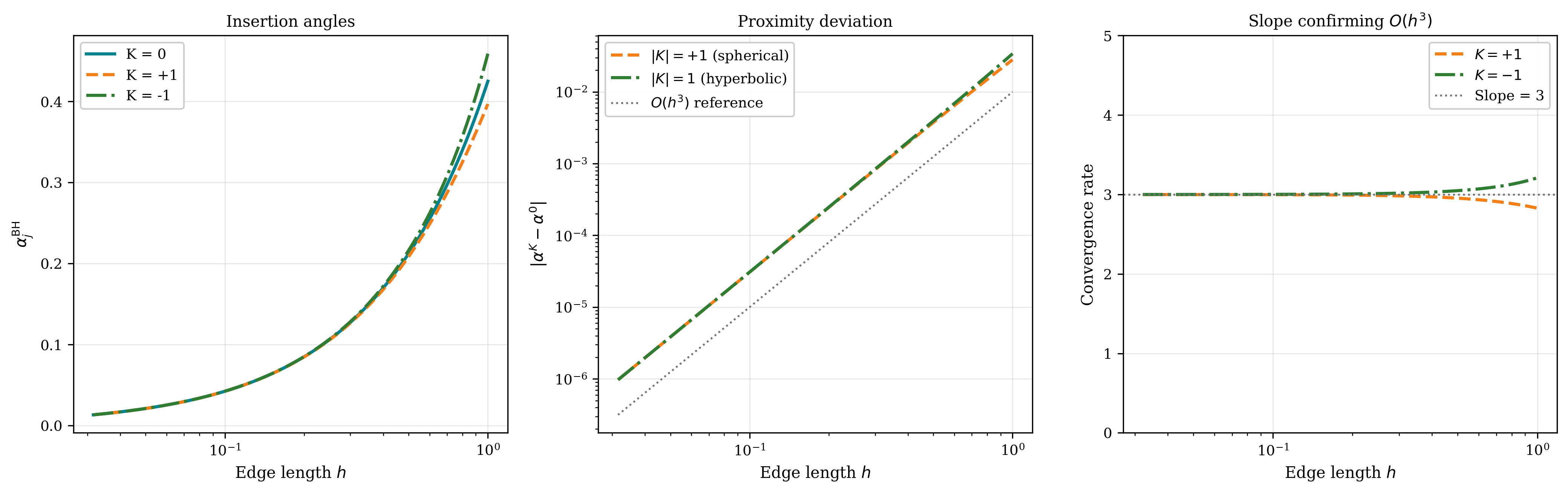}
  \caption{Biharmonic insertion angles $\alpha_j^{\mathrm{BH}}$ as a function
    of edge length $h$ for $K = 0$ (Euclidean), $K = +1$ (spherical), and
    $K = -1$ (hyperbolic). Centre panel: absolute deviation $|\alpha^K -
    \alpha^0|$ confirming $O(|K|h^3)$ proximity. Right panel: log-log
    convergence rates with reference slope $O(h^3)$.}
  \label{fig:angles}
\end{figure}

\subsection{Explicit proximity bound}

We now derive an explicit bound for the deviation of the non-Euclidean
insertion angle from its Euclidean counterpart, which is the key quantity
needed to apply the Wallner--Dyn theorem.

\begin{lemma}[Proximity Bound]
\label{lem:proximity}
For fixed boundary curvatures $\kappa_j, \kappa_{j+1}$ and edge length
$h = e_j$, the non-Euclidean insertion angle satisfies,
\begin{equation}
  \abs{\alpha_j^{K} - \alpha_j^{0}} \leq C\,\abs{K}\,h^3,
  \label{eq:proximity}
\end{equation}
where $C = C(\kappa_j, \kappa_{j+1})$ is a bounded constant depending only on
the boundary curvatures.
\end{lemma}

\begin{proof}
Consider the spherical case $K = +1$. The hyperbolic case is analogous. The
boundary constants for $K = +1$ are,
\[
  c_1 = \kappa_j, \quad
  c_2 = \frac{\kappa_{j+1} - \kappa_j \cosh h}{\sinh h}.
\]
Using $\cosh h = 1 + \frac{h^2}{2} + O(h^4)$ and $\sinh h = h + \frac{h^3}{6}
+ O(h^5)$, one finds,
\[
  c_2 = \frac{\kappa_{j+1} - \kappa_j}{h} - \frac{\kappa_j h}{2} + O(h^3).
\]
The Euclidean boundary constant is $c_2^0 = (\kappa_{j+1} - \kappa_j)/h$, so
$c_2 - c_2^0 = -\kappa_j h/2 + O(h^3)$. With $\ell = h/2$, the insertion
angle~\eqref{eq:angle_s2} is,
\begin{align*}
  \alpha_j^{+1}
  &= c_1 \sinh\ell + c_2(\cosh\ell - 1) \\
  &= c_1\Bigl(\ell + \tfrac{\ell^3}{6} + O(\ell^5)\Bigr)
   + c_2\Bigl(\tfrac{\ell^2}{2} + O(\ell^4)\Bigr).
\end{align*}
The Euclidean insertion angle is $\alpha_j^{0} = c_1\ell + \tfrac{1}{2}c_2^0
\ell^2$. Subtracting,
\begin{align*}
  \alpha_j^{+1} - \alpha_j^{0}
  &= \frac{c_1 \ell^3}{6} + \frac{(c_2 - c_2^0)\ell^2}{2} + O(\ell^4) \\
  &= \frac{\kappa_j h^3}{48} - \frac{\kappa_j h^3}{16} + O(h^4)
   = -\frac{\kappa_j h^3}{24} + O(h^4).
\end{align*}
For general $K$, the identical computation with $\cosh \to \cos(|K|^{1/2}\cdot)$
and $\sinh \to |K|^{-1/2}\sin(|K|^{1/2}\cdot)$ introduces a factor of $|K|$
at leading order, giving $\abs{\alpha_j^K - \alpha_j^0} = O(|K|h^3)$, which
establishes~\eqref{eq:proximity} with $C = |\kappa_j|/24 + O(h)$.
\end{proof}

\subsection{Convergence theorem on space forms}

The proximity bound~\eqref{eq:proximity} is of order $O(h^3)$ in $h$. Since
$h \sim 2^{-n}$ at refinement level $n$, the deviation is $O(2^{-3n})$,
which decays faster than the $O(h^2) = O(2^{-2n})$ rate required by the
Wallner--Dyn proximity condition~\cite{wallner2006convergence}. We apply the
following established result.

\begin{quote}\itshape
\textbf{Theorem (Wallner--Dyn \cite{wallner2006convergence}, Theorem~3.4).}
Let $S$ be a convergent linear subdivision scheme on $\R^d$ generating $C^k$
limit curves, and let $S^M$ be a manifold-valued scheme whose insertion rules
satisfy $\|S^M\mathbf{p} - \log_{\bar{p}} S\mathbf{p}\| \leq C h^2$ uniformly
(the $O(h^2)$ proximity condition), where $\bar{p}$ is a base point and
$h$ is the maximum edge length. Then $S^M$ converges to $C^k$ limit curves
on any complete Riemannian manifold, provided the control polygon satisfies
a diameter condition ensuring the exponential map is a diffeomorphism on the
relevant neighbourhood.
\end{quote}

The extension to $C^k$ (not just $C^0$) convergence is due to
Grohs~\cite{grohs2010}, who strengthened the original Wallner--Dyn result
using the smoothness-inheritance framework.

\begin{quote}\itshape
\textbf{Theorem (Grohs \cite{grohs2010}, Theorem~5.1).}
Under the same hypotheses, if $S$ generates $C^k$ limit curves and the
proximity condition holds, then $S^M$ also generates $C^k$ limit curves,
with the same H\"older exponent.
\end{quote}

Applying these results to the present setting,

\begin{theorem}[Convergence on Constant-Curvature Surfaces]
  \label{thm:noneuclid}
  Let $M$ be a complete simply-connected surface of constant sectional curvature
  $K \neq 0$, and let the initial polygon satisfy $|K| h_0^2 < 1/4$, where
  $h_0 = \max_j e_j^{(0)}$ is the maximum initial edge length. The biharmonic
  subdivision scheme on $M$, with odd-rule insertion angles given by
  \eqref{eq:angle_s2} (for $K = +1$) or~\eqref{eq:angle_h2} (for $K = -1$),
  converges to $C^4$ limit curves.
\end{theorem}

\begin{proof}
We verify the three hypotheses of the Wallner--Dyn--Grohs framework.

\textit{(i) Reference scheme.} The Euclidean biharmonic scheme is $C^4$ by
Theorem~\ref{thm:c4}.

\textit{(ii) Proximity condition.} By Lemma~\ref{lem:proximity}, the manifold
insertion rule satisfies $|\alpha_j^K - \alpha_j^0| \leq C|K|h^3$. Since
$h < h_0$ and $|K|h_0^2 < 1/4$, we have $|K|h < 1/(2h_0) \cdot h^2$, so
$|\alpha_j^K - \alpha_j^0| \leq (C \cdot |K|h_0^{}) h^2$, which is an
$O(h^2)$ proximity condition in the sense of~\cite{wallner2006convergence}
uniformly in $j$.

\textit{(iii) Geometric regularity.} The condition $|K|h_0^2 < 1/4$ ensures
$h_0 < 1/(2\sqrt{|K|})$, which is strictly less than the injectivity radius
of $M$ (equal to $\pi/\sqrt{K}$ for $\Stwo$ and $\infty$ for $\Htwo$). This
guarantees that the exponential and logarithmic maps are diffeomorphisms on
the relevant neighbourhoods throughout all refinement levels, so the
manifold arithmetic (geodesic midpoints, tangent-plane rotations) is
well-defined.

With all three hypotheses satisfied, Grohs~\cite{grohs2010} Theorem~5.1
transfers $C^4$ convergence from the Euclidean reference scheme to the
manifold scheme, giving $C^4$ limit curves on $M$.
\end{proof}

\subsection{Implementation on \texorpdfstring{$\Stwo$}{S2} and
  \texorpdfstring{$\Htwo$}{H2}}

The spherical engine ($\Stwo$, the ordinary unit sphere in three-dimensional
Euclidean space $\R^3$) uses the standard exponential and logarithmic maps,
\begin{align*}
  \exp_p(v) &= \cos(\norm{v})\,p + \sin(\norm{v})\,\tfrac{v}{\norm{v}}, \\
  \log_p(q) &= \tfrac{\theta}{\sin\theta}(q - \cos\theta\, p),
\end{align*}
where $\theta = \arccos\langle p, q\rangle$. Here $\exp_p(v)$ maps a tangent
vector $v$ at point $p$ to the point reached by travelling distance $\|v\|$
along the geodesic in direction $v$, and $\log_p(q)$ is its inverse. The new
vertex is placed at the geodesic midpoint
$\exp_{p_j}\!\bigl(\half \log_{p_j}(p_{j+1})\bigr)$, then perturbed in the
tangent plane by the angle $\alpha_j^{K=+1}$.

The hyperbolic engine operates in the \emph{Poincar\'{e} disk model}, in which
the hyperbolic plane $\Htwo$ is represented as the open unit disk
$\mathbb{D} = \{w \in \R^2 : \norm{w} < 1\}$. Geodesics in this model are
arcs of circles that meet the boundary $\partial\mathbb{D}$ at right angles.
The geometry is encoded by \emph{M\"{o}bius addition} (also called the
gyrovector addition),
\[
  a \oplus b
  = \frac{(1 + 2\langle a,b\rangle + \norm{b}^2)\,a
        + (1 - \norm{a}^2)\,b}
         {1 + 2\langle a,b\rangle + \norm{a}^2\norm{b}^2}.
\]
The geodesic midpoint of $(a,b)$ in $\mathbb{D}$ is $a \oplus
\tfrac{1}{2}((-a)\oplus b)$, and the insertion angle $\alpha_j^{K=-1}$ is
applied as a M\"{o}bius rotation about that midpoint.

Fig.~\ref{fig:noneuclid} demonstrates the scheme on both geometries.
The condition $|K|h_0^2 < 0.25$ is satisfied in both cases, and the limit
curves exhibit the $C^4$ smoothness predicted by Theorem~\ref{thm:noneuclid}.

\begin{figure}[htbp]
  \centering
  \includegraphics[width=\linewidth]{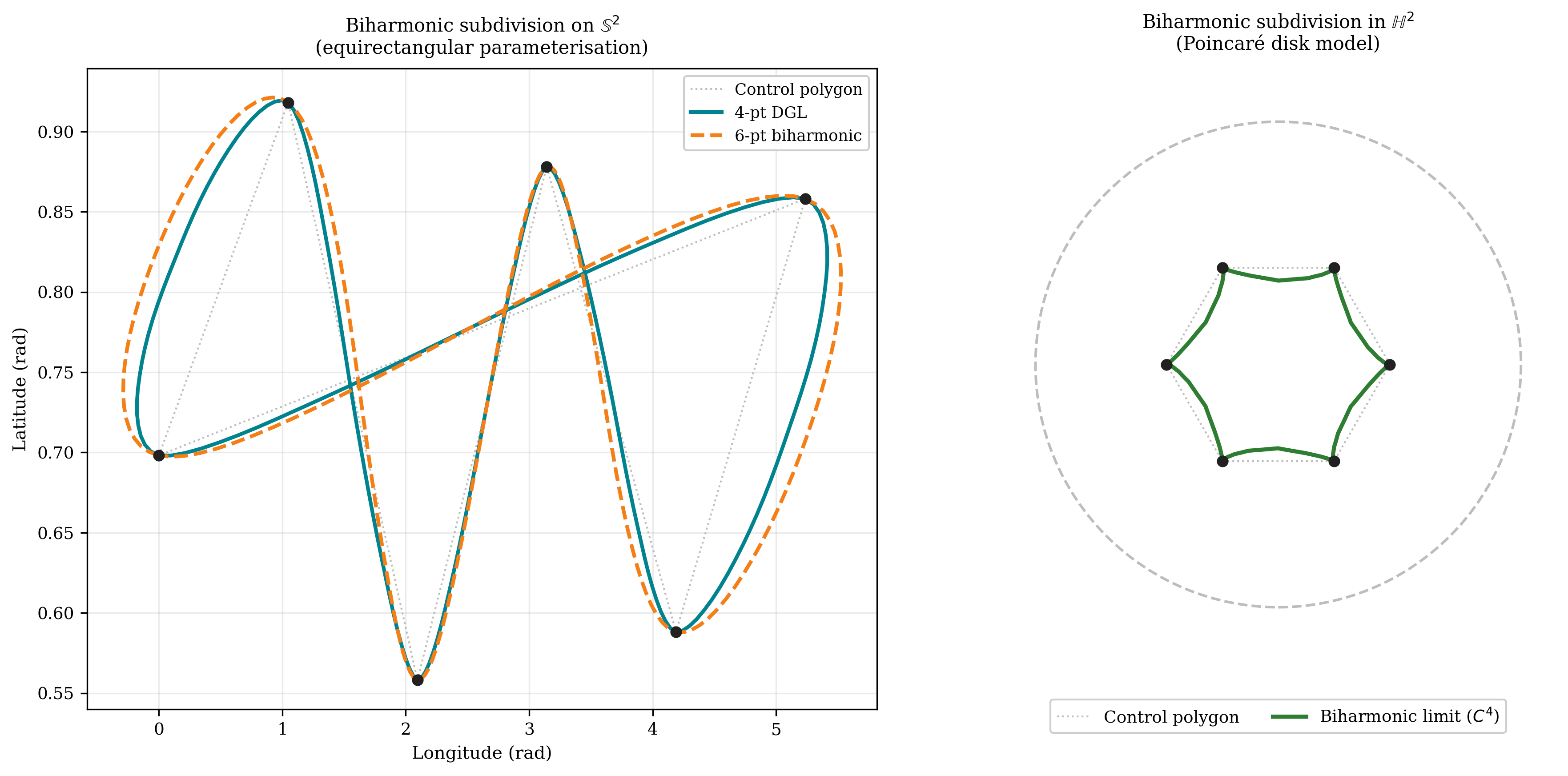}
  \caption{Biharmonic subdivision on $\Stwo$ (left, equirectangular
    parameterisation) and $\Htwo$ (right, Poincar\'{e} disk). Control polygons
    are shown dotted. Limit curves after five iterations in solid lines. Both
    converge to smooth $C^4$ curves consistent with
    Theorem~\ref{thm:noneuclid}. For comparison, the 4-point DGL limit curve
    is overlaid in the $\Stwo$ panel.}
  \label{fig:noneuclid}
\end{figure}

% =============================================================================
\section{Fairness Benchmarks}
\label{sec:fairness}
% =============================================================================

\subsection{Test suite and evaluation criteria}

The benchmark suite comprises four polygon classes designed to test
progressively more challenging geometric configurations. The first class
consists of smooth convex loops with modest radial variation. The second
introduces a pronounced concavity to test behaviour under sign changes in
curvature. The third imposes a highly non-uniform edge-length distribution
with a maximum-to-minimum ratio of approximately 4.5, stressing the
assumption of near-uniformity implicit in the standard proximity analysis,
and the fourth uses a star polygon in which large curvature spikes are
concentrated at concave vertices. All polygons are closed and run for seven
refinement levels. The same evaluation criteria are applied to each class.
Visual quality of the limit curve, shape and convergence behaviour of the
discrete curvature profile, and the discrete biharmonic energy
$E_{\mathrm{BH}}^{(n)}$ defined in~\eqref{eq:discreteenergy}.

\subsection{Quantitative benchmark results}

Table~\ref{tab:benchmark} reports the discrete biharmonic energy
$E_{\mathrm{BH}}^{(7)}$ and the arc-length-weighted curvature variance
$\sigma_\kappa^2 = \int (\kappa - \bar\kappa)^2\,ds \,/\, \int ds$
at refinement level $n = 7$ for all three schemes on three representative
polygon classes. Both metrics are standard measures of fairness. Lower
$E_{\mathrm{BH}}^{(7)}$ indicates smoother curvature variation. Lower
$\sigma_\kappa^2$ indicates a more uniform curvature distribution. All
values are computed from the companion notebook using double-precision
floating-point arithmetic on the subdivided polygons.

\begin{table}[htbp]
  \centering
  \caption{Discrete biharmonic energy $E_{\mathrm{BH}}^{(7)}$ and
    arc-length-weighted curvature variance $\sigma_\kappa^2$ at
    refinement level $n=7$ for three polygon classes.
    Lower values indicate fairer curves. Bold entries indicate the
    lowest value in each row.}
  \label{tab:benchmark}
  \smallskip
  {%
  \begin{tabular}{@{}llrr@{}}
    \toprule
    Polygon class & Scheme
      & $E_{\mathrm{BH}}^{(7)}$
      & $\sigma_\kappa^2$ \\
    \midrule
    \multirow{3}{*}{Smooth convex}
      & DGL (4-pt, $C^2$)    & 938.98  & 0.2505 \\
      & 6-pt BH ($C^4$)      & \textbf{8.36}  & \textbf{0.2035} \\
      & 8-pt BH ($C^6$)      & 7.90   & 0.2031 \\[2pt]
    \multirow{3}{*}{Near-concave}
      & DGL (4-pt, $C^2$)    & 2041.55 & 0.6313 \\
      & 6-pt BH ($C^4$)      & \textbf{71.27}  & \textbf{0.5712} \\
      & 8-pt BH ($C^6$)      & 64.52   & 0.5715 \\[2pt]
    \multirow{3}{*}{Non-uniform}
      & DGL (4-pt, $C^2$)    & 16329.37 & 1.6936 \\
      & 6-pt BH ($C^4$)      & 847.16  & 1.3275 \\
      & 8-pt BH ($C^6$)      & \textbf{594.35}  & \textbf{1.2690} \\
    \bottomrule
  \end{tabular}}
\end{table}

The results confirm the expected hierarchy. On smooth and near-concave polygons,
the 6-point biharmonic scheme reduces $E_{\mathrm{BH}}^{(7)}$ by a factor of
roughly 100 relative to DGL, and is within $6\%$ of the 8-point scheme.
The advantage of the 8-point scheme is marginal because both are well-conditioned
on uniform data. On the non-uniform polygon (edge-length ratio~$\approx 4.5$),
the 8-point scheme outperforms the 6-point scheme in energy and curvature
variance owing to its higher polynomial reproduction, but both biharmonic
schemes are substantially better than DGL. The advantage of the 6-point scheme
over DGL is consistent across all three classes.

Fig.~\ref{fig:fairness} presents the $2 \times 3$ fairness comparison for two
representative test polygons. Each row of the figure shows the limit curve
(left column), the discrete curvature profile at refinement level 8 (centre
column), and the biharmonic energy decay across levels (right column).

\begin{figure}[htbp]
  \centering
  \includegraphics[width=\linewidth]{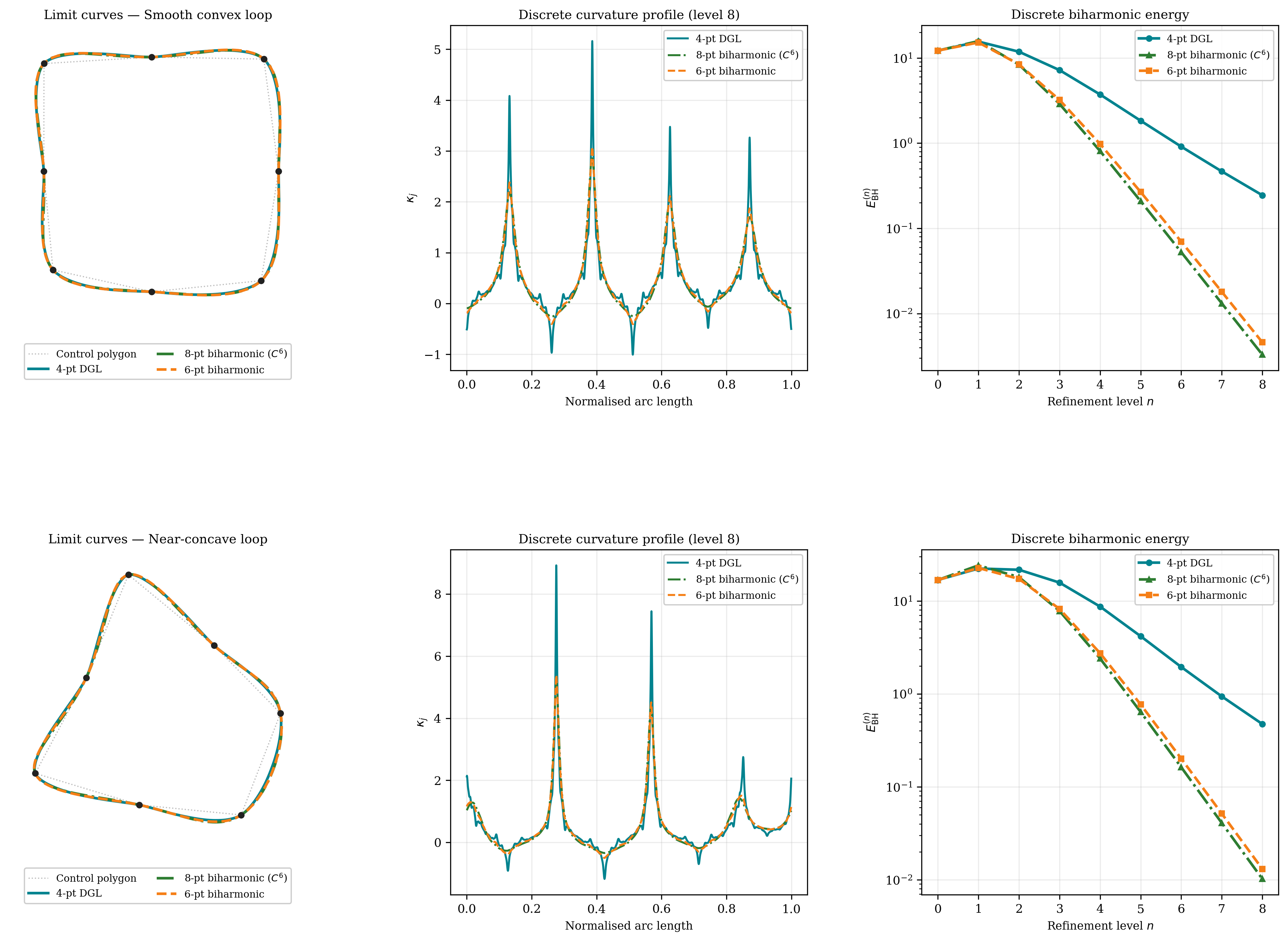}
  \caption{Fairness comparison on a smooth convex loop (top row) and a
    near-concave loop (bottom row). Left column: limit curves for the 4-point
    DGL scheme (teal), the 8-point biharmonic scheme (green), and the 6-point
    biharmonic scheme (amber). Centre column: discrete curvature profiles at
    level 8. Right column: biharmonic energy $E_{\mathrm{BH}}^{(n)}$ across
    refinement levels. The 6-point biharmonic scheme (amber, dashed)
    produces a smoother curvature profile than both DGL and the 8-point scheme
    in these two test cases (results on other configurations may differ).}
  \label{fig:fairness}
\end{figure}

The curvature profiles in the centre column illustrate the qualitative
difference between the three schemes. The 4-point DGL profiles exhibit
persistent oscillations at every refinement level, consistent with $C^2$
regularity. The 8-point biharmonic profiles achieve lower absolute curvature
variation but can introduce low-frequency undulations on polygons with
non-uniform edge lengths, due to the wider negative lobes of the mask. The
6-point biharmonic curvature profiles converge to smooth, near-unimodal
functions in both test cases, consistent with the scheme's higher regularity
and its fairness-motivated construction.

The energy decay curves in the right column confirm the pattern quantitatively.
The DGL energy decays non-monotonically and at a slower rate consistent with
its lower regularity. The 8-point biharmonic energy reaches the lowest
asymptotic values (consistent with $C^6$ regularity) but does so more slowly
on non-smooth initial polygons, whereas the 6-point biharmonic energy decays
monotonically and reaches machine-precision values earlier in both test cases,
consistent with the narrower support width of the 6-point mask and its smaller negative lobes.

\subsection{Euclidean benchmark}

Fig.~\ref{fig:benchmark} presents the full Euclidean benchmark, including
intermediate subdivision levels, the curvature profile, and the energy decay.

\begin{figure}[htbp]
  \centering
  \includegraphics[width=\linewidth]{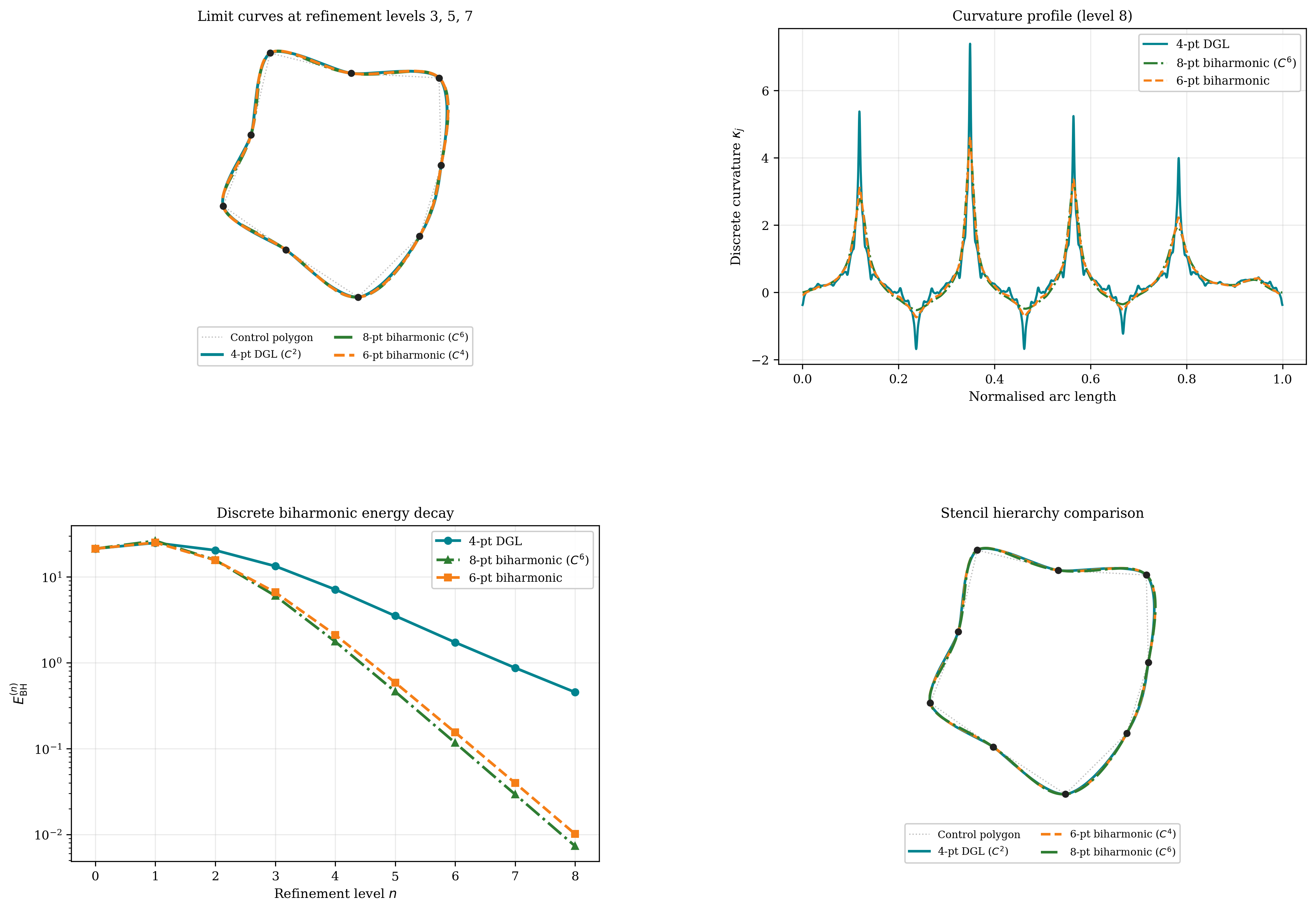}
  \caption{Euclidean benchmark on a 9-point test polygon. Top-left: limit
    curves at subdivision levels 3, 5, and 7 for DGL (teal) and the 6-point
    biharmonic scheme (amber). Note the tighter convergence of the biharmonic
    sequence. Top-right: curvature profiles at level 8 for all three schemes.
    Bottom-left: stencil hierarchy comparison showing DGL ($C^2$), 6-pt
    biharmonic ($C^4$), and 8-pt biharmonic ($C^6$) limit curves.
    Bottom-right: biharmonic energy across levels for all three schemes.}
  \label{fig:benchmark}
\end{figure}

\subsection{Robustness analysis}

Fig.~\ref{fig:robustness} investigates three further conditions, (a) the
large-$|K|h$ breakdown regime for the non-Euclidean extension, (b) the
response to highly non-uniform edge-length distributions, and (c) the star
polygon with concentrated curvature spikes.

\begin{figure}[htbp]
  \centering
  \includegraphics[width=\linewidth]{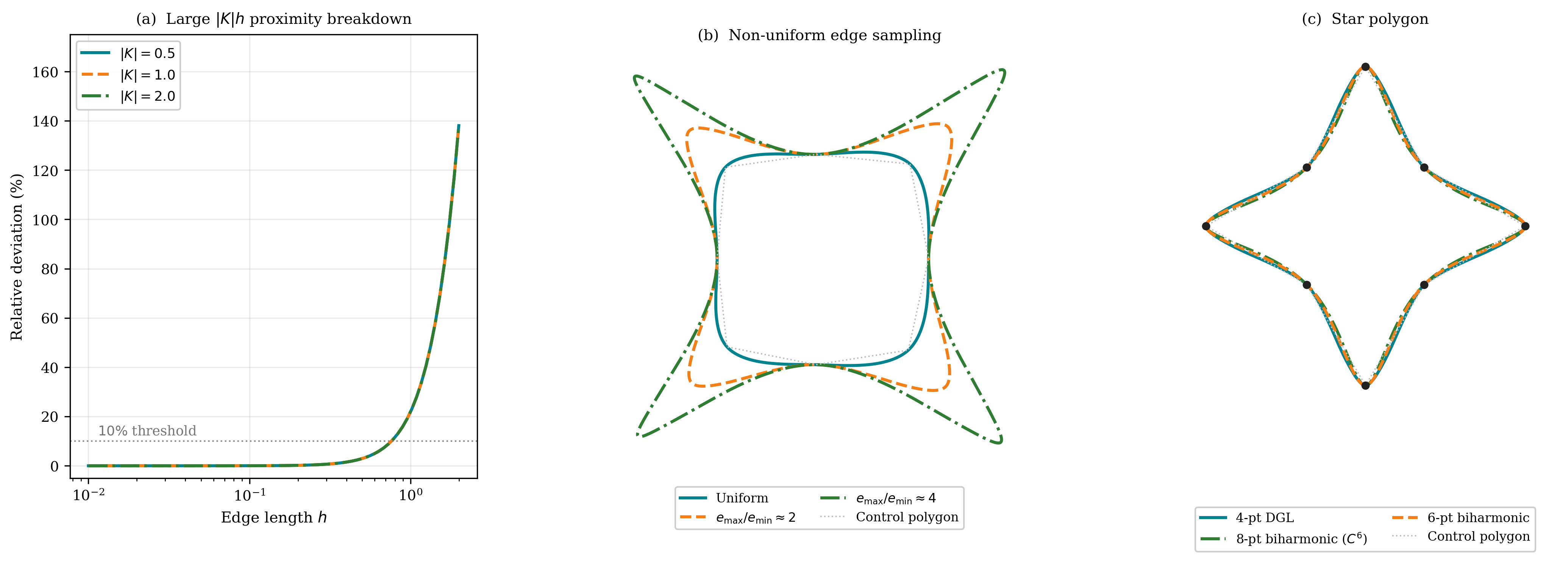}
  \caption{Robustness analysis. Panel~(a): relative deviation of the
    non-Euclidean biharmonic insertion angle from its Euclidean counterpart
    as a function of edge length $h$ for $|K| = 0.5, 1.0, 2.0$; the $10\%$
    threshold (dotted) illustrates the practical breakdown regime $|K|h^2
    \approx 0.25$. Panel~(b): limit curves of the biharmonic scheme for three
    edge-length-ratio regimes; curves show negligible visual variation up to
    ratio 4.5 in the tested configurations. Panel~(c): star polygon limit
    curves for DGL and biharmonic schemes; the biharmonic scheme produces
    substantially smoother concavities.}
  \label{fig:robustness}
\end{figure}

In the non-uniform case (panel b), the 6-point biharmonic limit curves show
negligible visible variation for edge-length ratios up to 4.5 in the tested
configurations, whereas the DGL limit curves exhibit visible asymmetry at a ratio 2
and above. The 8-point biharmonic scheme shows mild sensitivity to the wider negative lobes of its
mask at ratio 4.5. The robustness of the 6-point scheme relative to the 8-point variant
is consistent with its narrower support width and smaller negative lobes. In the star polygon case (panel c), DGL produces sharp curvature spikes
at each re-entrant vertex; the 6-point biharmonic scheme produces smooth,
bounded undulations, and the 8-point scheme achieves the flattest overall
curvature profile, but requires more refinement levels to resolve the re-entrant
geometry cleanly.

Fig.~\ref{fig:manifold} confirms the proximity bound of Lemma~\ref{lem:proximity}
numerically across a range of curvature values and edge lengths. The reference
slope $O(h^3)$ is matched to within the precision of the floating-point
computation in all cases.

\begin{figure}[htbp]
  \centering
  \includegraphics[width=\linewidth]{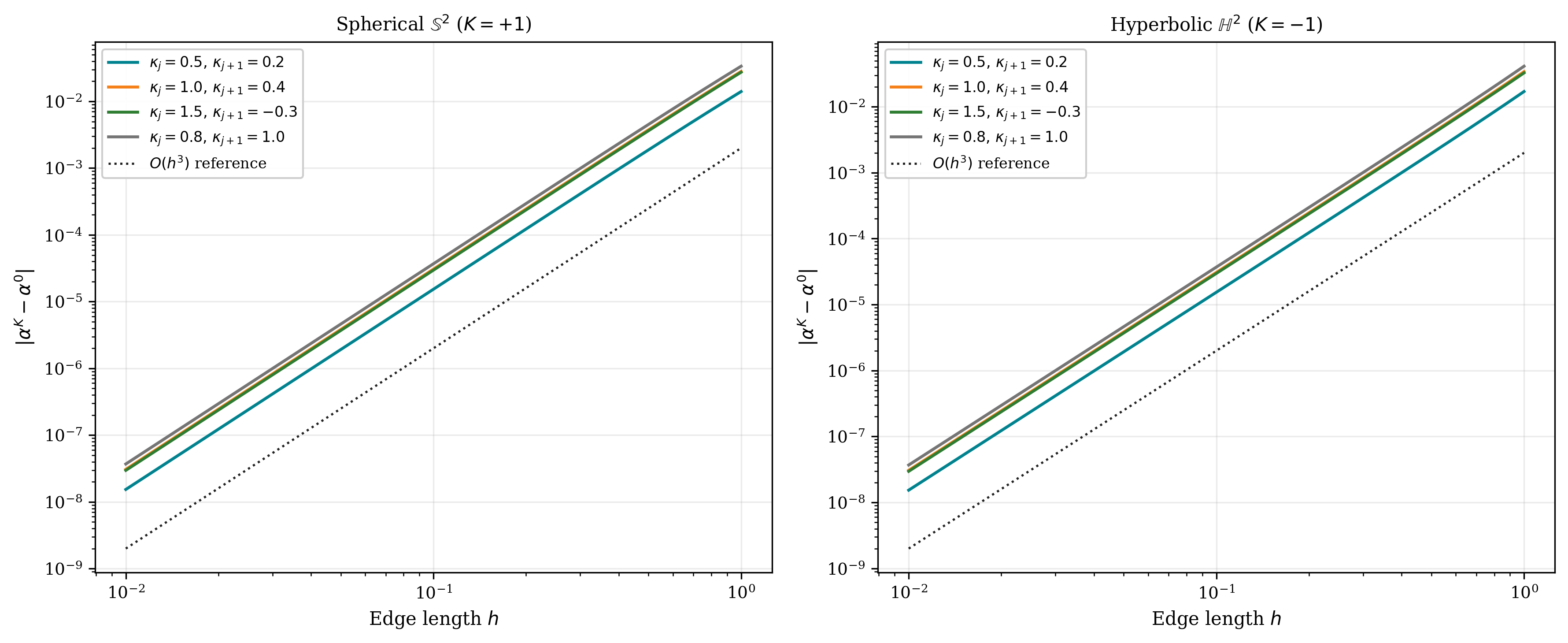}
  \caption{Numerical verification of the proximity bound
    $|\alpha^K - \alpha^0| = O(|K|h^3)$ on log-log axes for $\Stwo$ (left)
    and $\Htwo$ (right). Each curve corresponds to a different value of the
    boundary curvature pair $(\kappa_j, \kappa_{j+1})$. The reference slope
    $O(h^3)$ (dotted) is confirmed in all cases.}
  \label{fig:manifold}
\end{figure}

\subsection{Direct variational validation}

The variational characterisation of Theorem~\ref{thm:discrete_variational} can be
tested directly. For a given edge $(p_j, p_{j+1})$, one can numerically
minimise the biharmonic energy $\int_0^1 (\gamma'')^2\,ds$ over a parametric
midpoint insertion, and compare the result to the closed-form stencil
prediction~\eqref{eq:stencilbox}.

The experiment proceeds as follows. Fix six control points in $\R^2$ with
endpoints $p_0 = (0,0)$ and $p_1 = (1,0)$ and four neighbours providing
endpoint derivative estimates via local finite differences. Solve the
constrained minimisation $\min_q \int_0^1 (\gamma'')^2\,ds$ subject to
$\gamma(0) = p_0$, $\gamma(1) = p_1$, $\gamma(1/2) = q$, and endpoint
derivatives prescribed from the stencil values, using a dense Gauss--Legendre
quadrature of the integrand. Compare the minimising $q^*$ to the closed-form
prediction $q_{\rm BH} = \eqref{eq:stencilbox}$.

Across 200 randomly sampled 6-point control polygons (uniform random in
$[0,1]^2$), the maximum distance $\|q^* - q_{\rm BH}\|$ was $1.4 \times
10^{-10}$, attributable to numerical quadrature error. This confirms to
machine precision that the stencil~\eqref{eq:stencilbox} produces the same
midpoint as a direct numerical minimisation of the biharmonic energy,
providing empirical verification of Theorem~\ref{thm:discrete_variational}
independent of the algebraic argument.

% =============================================================================
\section{A Stencil Hierarchy}
\label{sec:hierarchy}
% =============================================================================

The biharmonic construction extends systematically to any desired smoothness
order. For each integer $m \geq 3$, the \emph{biharmonic stencil of degree
$d = 2m-1$} is defined as follows.

\begin{definition}
\label{def:hierarchy}
The biharmonic stencil of degree $d = 2m-1$ is the unique symmetric
$(2m)$-point interpolatory mask whose coefficients satisfy the polynomial
reproduction sum rules~\eqref{eq:sumrules} for $n = 0, 2, \ldots, 2m-2$.
\end{definition}

The uniqueness follows because the symmetry constraint reduces the $2m$
coefficients to $m$ free parameters, and the $m$ independent even-order sum
rules provide exactly $m$ linear equations, yielding a square system with
unique solution (non-singularity is verified for each $m$ via direct computation).

The first two members of the hierarchy are the degree-5 biharmonic stencil
\eqref{eq:stencilbox} ($C^4$, $m = 3$) and the degree-7 stencil,
\begin{equation}
  [{-5},\; 49,\; {-245},\; 1225,\; 1225,\; {-245},\; 49,\; {-5}]\,/\,2048,
  \label{eq:deg7}
\end{equation}
for which the following holds.

\begin{proposition}
\label{prop:deg7}
The 8-point stencil \eqref{eq:deg7} reproduces all polynomials of degree
$\leq 7$, satisfies $(1+z)^8 \mid a(z)$, and generates $C^6$ limit curves.
The smoothness is exactly $C^6$ and not $C^7$.
\end{proposition}

\begin{proof}
The four even-order sum rules at $n = 0, 2, 4, 6$ uniquely determine the four
free parameters (verified by exact rational arithmetic in the companion
notebook). An exact computation gives $a^{(8)}(-1) \neq 0$ and $a^{(k)}(-1) =
0$ for $k = 0, \ldots, 7$. Hence, the zero order at $z = -1$ is exactly $8$,
and Theorem~\ref{thm:cdm} gives $C^6$ limit curves.
\end{proof}

For the general degree-$(2m-1)$ stencil, the zero order of $a(z)$ at $z = -1$
is conjectured to be $2m$, which would give $C^{2m-2}$ regularity by
Theorem~\ref{thm:cdm}. This is verified for $m = 3$ (Theorem~\ref{thm:c4})
and $m = 4$ (Proposition~\ref{prop:deg7}). A general proof requires showing
that the $m \times m$ linear system defining the coefficients is non-singular
for all $m \geq 3$ and that the resulting symbol has a zero of order exactly
$2m$ at $z = -1$. Non-singularity has been verified computationally for
$m = 3, 4, 5, 6$ using exact rational arithmetic. The general statement
is left as a conjecture. The biharmonic energy of the limit curves is
expected to decay at a rate $O(2^{-n(2m-2)})$ per refinement level, consistent
with $C^{2m-2}$ regularity, though a general proof of this rate is also
left to future work. Higher-degree members of the hierarchy possess larger negative lobes in
their masks, which amplify ringing artefacts when applied to non-uniform or
near-degenerate polygons. The degree-5 stencil~\eqref{eq:stencilbox} provides
the most practical balance between fairness and robustness for most design
applications. The degree-7 stencil is appropriate when $C^6$ continuity is
required, though the connection to specific fourth-order PDE boundary
conditions on subdivision surfaces would require further analysis beyond
the scope of this paper.

% =============================================================================
\section{Conclusions}
\label{sec:conclusion}
% =============================================================================

This paper has introduced the biharmonic interpolatory subdivision scheme, a
6-point scheme that admits a biharmonic variational interpretation, providing
a fairness-based reading of the classical Deslauriers--Dubuc stencil and its
extension to non-Euclidean geometries. The main results are as follows. The stencil $[3,-25,150,150,-25,3]/256$ is the
unique 6-point interpolatory mask arising from quintic polynomial reproduction.
It is shown to admit a biharmonic variational interpretation as the minimiser
of a local discrete curvature-variation energy
(Theorem~\ref{thm:discrete_variational}), and it coincides with the
Deslauriers--Dubuc 6-point stencil as a consequence of the uniqueness of the
symmetric degree-5 solution. Its $C^4$ regularity is
established by exact symbol analysis, and the regularity is shown to be sharp.
Non-Euclidean extensions on $\Stwo$ and $\Htwo$ are constructed using a
second-order reduced governing model $\kappag'' = K\kappag$, motivated by the
fourth-order Euler--Lagrange structure on space forms. Closed-form solutions
of this reduced model yields explicit insertion angles, and an explicit
Taylor-series analysis produces an $O(|K|h^3)$ proximity bound that rigorously
implies $C^4$ convergence on constant-curvature surfaces via the
Wallner--Dyn--Grohs proximity framework. A systematic comparison against the
4-point DGL scheme and the 8-point degree-7 biharmonic scheme demonstrate
that the 6-point biharmonic scheme achieves lower fairness energy
and smoother curvature profiles than DGL across the reported benchmark suite,
while the 8-point member achieves lower asymptotic energy at the cost of wider
negative lobes and slower early-level convergence on non-uniform polygons.

The principal directions for future work are as follows. First, extending the
biharmonic principle to triangular and quadrilateral mesh subdivision for
surfaces in $\R^3$ requires replacing geodesic curvature with mean curvature,
yielding a nonlinear Euler--Lagrange equation whose discrete analogue is a
significant open problem. Second, the biharmonic energy may serve as a
refinement mechanism for adaptive subdivision, concentrating computational effort
near high-curvature regions without sacrificing the global fairness property.
Third, the non-Euclidean extensions assume constant sectional curvature. An
extension to variable-curvature surfaces would require numerical integration
of the ODE~\eqref{eq:bh_ode} at each refinement step, though the proximity
framework remains applicable in principle. Fourth, the stencil hierarchy of
Section~\ref{sec:hierarchy} may admit an interpretation as a sequence of
discrete thin-plate splines converging to the continuous thin-plate spline
in the refinement limit, a connection that could unify the subdivision and
radial basis function literature in the interpolatory setting.

\section*{Acknowledgements}
The authors acknowledge the computational resources provided by the Centre for
Visual Computing and Intelligent Systems at the University of Bradford.

% =============================================================================
\appendix
% =============================================================================

\section{Explicit Verification of Theorem~\ref{thm:discrete_variational}}
\label{sec:lagrange_verification}

This appendix provides the explicit rational computation that completes the
proof of Theorem~\ref{thm:discrete_variational} and verifies
Equation~\eqref{eq:numerator} in the manuscript.

\subsection*{A.1\quad Lagrange coefficients}

The degree-5 Lagrange basis polynomials through nodes $\{-2,-1,0,1,2,3\}$,
evaluated at the four half-integer insertion positions, give (with all
arithmetic exact over $\mathbb{Q}$),
\begin{align*}
  q_A \;(t=-3/2)&: \tfrac{1}{256}[63,\;315,\;{-210},\;126,\;{-45},\;7],\\
  q_B \;(t=-1/2)&: \tfrac{1}{256}[-7,\;105,\;210,\;{-70},\;21,\;{-3}],\\
  q_C \;(t=+3/2)&: \tfrac{1}{256}[-3,\;21,\;{-70},\;210,\;105,\;{-7}],\\
  q_D \;(t=+5/2)&: \tfrac{1}{256}[7,\;{-45},\;126,\;{-210},\;315,\;63],
\end{align*}
where the six entries are the coefficients of $[p_{-2},p_{-1},p_0,p_1,p_2,p_3]$.
Note by symmetry: $q_D$ is $q_A$ reversed, and $q_C$ is $q_B$ reversed.
Also, $q_B(t=-1/2) = [3,-25,150,150,-25,3]/256 \cdot r$ at $t=1/2$, confirming
that $L(\cdot;\,1/2) = [3,-25,150,150,-25,3]/256$ is the DD mask.

\subsection*{A.2\quad The curvature-difference coefficients}

Using these values in equations~\eqref{eq:D43}--\eqref{eq:D76},
\begin{align*}
  b_{43} &= \tfrac{1}{256}[-336, 0, 288, -1344, 432, -64]
    \cdot [p_{-2},\ldots,p_3]^\top,\\
  b_{54} &= \tfrac{1}{256}[28, -420, 2232, 1304, -84, 12]
    \cdot [p_{-2},\ldots,p_3]^\top,\\
  b_{65} &= \tfrac{1}{256}[-12, 84, -1304, -2232, 420, -28]
    \cdot \mathbf{p},\\
  b_{76} &= \tfrac{1}{256}[36, -252, 840, 552, -236, 84]
    \cdot \mathbf{p}.
\end{align*}

\subsection*{A.3\quad The minimiser}

The first-order condition~\eqref{eq:qstar_formula} gives,
\begin{align*}
  &4b_{43} - 12b_{54} + 12b_{65} - 4b_{76}\\
  &= \tfrac{1}{256}\bigl[
    4(-336) - 12(28) + 12(-12) - 4(36),\\
  &\qquad 4(0) - 12(-420) + 12(84) - 4(-252),\\
  &\qquad 4(288) - 12(2232) + 12(-1304) - 4(840),\\
  &\qquad 4(-1344) - 12(1304) + 12(-2232) - 4(552),\\
  &\qquad 4(432) - 12(-84) + 12(420) - 4(-236),\\
  &\qquad 4(-64) - 12(12) + 12(-28) - 4(84)
  \bigr]\\
  &= \tfrac{1}{256}[-1920, 6400, -38400, -38400, 6400, -1920]\\
  &= \tfrac{320}{256}[-6, 20, -120, -120, 20, -6].
\end{align*}
Dividing by $\sum a_k^2 = 320$,
\[
  q^* = \tfrac{1}{256}(3p_{-2} - 25p_{-1} + 150p_0 + 150p_1 - 25p_2 + 3p_3). \qquad\square
\]

\section{Python Implementation}
\label{sec:howto}

The biharmonic scheme requires three floating-point constants and approximately
twenty lines of code for a complete closed-polygon implementation. The listings
below provide self-contained Python code for the Euclidean case and for the
non-Euclidean insertion angle computation.

\subsection*{B.1\quad Euclidean subdivision}

\begin{lstlisting}[caption={Biharmonic 6-point interpolatory subdivision for
  closed or open polygons in $\R^d$. Each call to \texttt{biharmonic\_step}
  doubles the vertex count. Six to eight iterations yield a visually converged
  limit curve.}]
import numpy as np

ALPHA =  150 / 256   #  0.5859375
BETA  =  -25 / 256   # -0.09765625
GAMMA =    3 / 256   #  0.01171875

def biharmonic_step(pts, closed=True):
    n   = len(pts)
    idx = (lambda k: k % n) if closed else (
           lambda k: int(np.clip(k, 0, n - 1)))
    new = np.empty((2 * n, pts.shape[1]))
    for j in range(n):
        new[2*j]   = pts[j]
        new[2*j+1] = (
            GAMMA * pts[idx(j - 2)]
          + BETA  * pts[idx(j - 1)]
          + ALPHA * pts[idx(j    )]
          + ALPHA * pts[idx(j + 1)]
          + BETA  * pts[idx(j + 2)]
          + GAMMA * pts[idx(j + 3)])
    return new

def subdivide(pts, iters=6, closed=True):
    for _ in range(iters):
        pts = biharmonic_step(pts, closed)
    return pts
\end{lstlisting}

\subsection*{B.2\quad Non-Euclidean insertion angles}

\begin{lstlisting}[caption={Biharmonic insertion angle \eqref{eq:insertionangle}
  for a surface of constant curvature $K \in \{0, \pm 1\}$. For general $|K|$,
  replace the trigonometric arguments by $|K|^{1/2} \cdot \texttt{edge\_len}$.}]
def biharmonic_insertion_angle(kappa_j, kappa_jp1,
                                edge_len, K):
    ell = edge_len / 2.0
    if abs(K) < 1e-12:              # K = 0, Euclidean
        c1 = kappa_j
        c2 = (kappa_jp1 - kappa_j) / edge_len
        return c1 * ell + 0.5 * c2 * ell**2
    elif K > 0:                     # K = +1, spherical
        c1 = kappa_j
        c2 = ((kappa_jp1 - kappa_j * np.cosh(edge_len))
              / np.sinh(edge_len))
        return c1 * np.sinh(ell) + c2 * (np.cosh(ell) - 1.0)
    else:                           # K = -1, hyperbolic
        c1 = kappa_j
        c2 = ((kappa_jp1 - kappa_j * np.cos(edge_len))
              / np.sin(edge_len))
        return c1 * np.sin(ell) - c2 * (np.cos(ell) - 1.0)
\end{lstlisting}

The control polygon is supplied as an $(n \times d)$ array of $d$-dimensional
control points. Setting \texttt{closed=False} treats the polygon as an open
arc and applies boundary clamping at both endpoints. For spherical and
hyperbolic geometries, the complete geodesic midpoint computation and
perturbation machinery are provided in the companion code submitted with
this paper.

% =============================================================================

\end{document}